\documentclass[12pt]{article}
\usepackage{amsmath,amssymb,amsthm}
\newtheorem{thm}{Theorem}[section]

 \newtheorem{cor}{Corollary}[section]
 \newtheorem{lem}{Lemma}[section]
 \newtheorem{prop}{Proposition}[section]
 \newtheorem{defn}{Definition}[section]
\newtheorem{rem}{Remark}[section]

\begin{document}
\begin{center}
{\large{\bf Frequency-localization Duhamel principle and its application to the optimal decay of dissipative systems}}
\end{center}
\begin{center}
\footnotesize{Jiang Xu}\\[2ex]
\footnotesize{Department of Mathematics, \\ Nanjing
University of Aeronautics and Astronautics, \\
Nanjing 211106, P.R.China,}\\
\footnotesize{jiangxu\underline{ }79@nuaa.edu.cn}

\vspace{3mm}

\footnotesize{Faculty of Mathematics, \\ Kyushu University, Fukuoka 819-0395, Japan}\\

\vspace{10mm}

\footnotesize{Shuichi Kawashima}\\[2ex]
\footnotesize{Faculty of Mathematics, \\ Kyushu University, Fukuoka 819-0395, Japan,}\\
\footnotesize{kawashim@math.kyushu-u.ac.jp}\\
\end{center}
\vspace{6mm}

\begin{abstract}
Very recently, a new decay framework has been given by \cite{XK2} for linearized dissipative hyperbolic systems satisfying
the Kawashima-Shizuta condition on the framework of Besov spaces, which allows to pay less attention on the traditional spectral analysis. However, owing to interpolation techniques, the analysis remains valid only for nonlinear systems in higher dimensions $(n\geq3)$ and
the corresponding case of low dimensions was left open, which provides the main motivation of this work. Firstly, we develop
new time-decay properties on the frequency-localization Duhamel principle. Furthermore, it is shown that the classical solution and its derivatives of fractional order decay at the optimal algebraic rate in dimension 1, by time-weighted energy approaches in terms of low-frequency and high-frequency decompositions. Finally, as applications, decay results for several
concrete systems subjected to the same dissipative structure as general hyperbolic systems, for instance, damped compressible Euler equations, thermoelasticity with second sound and Timoshenko systems with equal wave speeds, are also obtained.
\end{abstract}


\section{Introduction}
As a continuous work of \cite{XK1,XK2}, we still consider a class of partial differential equations,  whose form is given by
\begin{equation}
U_{t}+\sum_{k=1}^{n}F^{k}(U)_{x_{k}}=G(U). \label{R-E1}
\end{equation}
Here $U=U(t,x)$ is the unknown $N$-vector valued function of $(t,x)\equiv(t,x_{1},x_{2},\cdot\cdot\cdot,x_{n})\in [0,+\infty)\times \mathbb{R}^{n}(n\geq1)$,
taking values in an open convex set $\mathcal{O}_{U}\subset
\mathbb{R}^{N}$ (called the state space). $F^{k}$ and $G$ are given
$N$-vector valued smooth functions on $\mathcal{O}_{U}$. The subscripts $t$ and $x_{k}$ refer to the partial derivatives with respect to $t$ and $x_{k}$, respectively.

In this paper, we focus on the Cauchy problem of (\ref{R-E1}), so the initial data are prescribed as
\begin{equation}
U_{0}=U(0,x),\ \
  x\in\mathbb{R}^{n}. \label{R-E2}
\end{equation}
It is well known that, if the source $G(U)\equiv0$, system (\ref{R-E1}) reduces to a system of conservation laws,
the main feature of which is the finite time blowup of classical solutions even when the initial data are smooth and small.
If $G(U)\neq0$, system (\ref{R-E1}) describes a great number of non-equilibrium physical process.
Important examples occur in chemically reactive flows, radiation hydrodynamics, invisicid gas dynamics with relaxation, nonlinear optics and so
on. See \cite{Y1} and reference cited therein. In these applications, $G(U)$ has, or can be transformed by a
linear transformation into, the following partial form
$$G(U)=\left(
 \begin{array}{c}
                                                                   0 \\
                                                                  g(U) \\
                                                                 \end{array}
                                                               \right),
$$
with $0\in \mathbb{R}^{N_{1}}, g(U)\in \mathbb{R}^{N_{2}}$, where
$N_{1}+N_{2}=N(N_{1}\neq0)$. Obviously, the source is not
present in all the components of the system. Due to the great potential for applications, mathematically, the system has been receiving increasing attention
over the past two decades. \textit{What conditions are posted on (\ref{R-E1}) such that it could prevent the finite time blowup of classical solutions at least for some restricted classes of initial data? If so, how does the solution decay in time with some explicit rates?} Let us briefly review previous efforts around these questions.

Chen, Levermore and Liu \cite{CLL} first presented a strictly convex entropy for (\ref{R-E1}), which is a natural extension of
the entropy due to Godunov \cite{G}, Friedrichs and Lax \cite{FL} for conservation laws. For (\ref{R-E1})  endowed with
the entropy function, an entropy dissipation condition was additionally assumed to develop a general result of global existence for small perturbations of equilibrium constant states by Hanouzet and Natalini \cite{HN} in one space dimension and Yong \cite{Y}
in several space dimensions, by assuming all the time the Kawashima-Shizuta condition. Subsequently, the second author and Yong \cite{KY}
gave a perfect definition for the dissipative entropy to establish the global existence of classical solutions without
the technical requirement assumed in \cite{HN,Y}. Furthermore, it was shown by
Ruggeri and Serre \cite{RS} that the constant equilibrium state $\bar{U}$ is time asymptotically $L^2$-stable for (\ref{R-E1}).
With the aid of a detailed analysis of the Green function for the linearized problem, Bianchini, Hanouzet and Natalini \cite{BHN}
showed the $L^p$-decay rates in the decay framework of $H^{s}\cap L^1$, which was first established in the doctoral thesis by
the second author \cite{Ka}, and great developed by Hoff and Zumbrun \cite{HZ} for compressible Navier-Stokes equations. Later, the second author and Yong \cite{KY2} employed time-weighted energy approaches initialled by Matsumura in \cite{Ma} and obtained the $L^p$ decay estimates, which pays less attention on the spectral analysis in comparison with \cite{BHN}.

It should be mentioned that above efforts achieved fell into the framework of local-in-time existence theory of Kato \cite{K} and Majda \cite{M}
for quasilinear hyperbolic systems, where the regularity index of spatially Sobolev spaces satisfies $s>1+n/2$.
Recently, the first and second authors investigated the critical index $s_{c}=1+n/2$, where the basic existence theory of Kato and Majda
fails. With the assumption of dissipative entropy in \cite{KY}, the local-in-time existence and blow-up criterion of classical solutions
was established in spatially critical Besov spaces. Furthermore,
an elementary fact that indicates the relation between homogeneous and inhomogeneous Chemin-Lerner spaces (mixed space-time Besov spaces) was well developed.
This fact enables us to capture the dissipation rates generated from the partial source term and obtain the general result of global-in-time existence in spatially critical Besov spaces, see \cite{XK1}.
As a continuation investigation, recently, we gave a new decay framework for linearized dissipative systems in $L^2(\mathbb{R}^{n})\cap\dot{B}^{-s}_{2,\infty}(\mathbb{R}^{n})$ (see \cite{XK2}), which generalized the classical one $L^2(\mathbb{R}^{n})\cap L^{p}(\mathbb{R}^{n})(1\leq p<2)$ in \cite{UKS}, since $L^{p}(\mathbb{R}^{n})\hookrightarrow\dot{B}^{-s}_{2,\infty}(\mathbb{R}^{n})(0<s\leq n/2)$. The key ingredient lies in employing the Littlewood-Paley decomposition on the dissipative structure
$${\rm Re}\,\lambda(i\xi)\leq -c\eta_1(\xi), \quad \mbox{with}\quad  \eta_1(\xi)=\frac{|\xi|^2}{1+|\xi|^2}$$
to obtain a differential inequality, which leads to the desired decay estimate. Actually, such manner goes considerably beyond the usual low-frequency and high-frequency integrals based on the Hausdorff-Young inequality as in \cite{Ka,UKS}. To obtain the corresponding decay estimates for nonlinear systems, we develop localized time-weighted energy approaches in terms of low-frequency and high-frequency decompositions, with the aid of frequency-localization Duhamel principle and improved Gagliardo-Nirenberg-Sobolev inequality (Lemma \ref{lem6.4}) that allows to the case of fractional derivatives. Due to techniques from interpolation inequalities,
the analysis remains valid only for the case of high dimensions $(n\geq3)$, and the optimal decay estimate in low dimensions was left open.

\subsection{Preliminary}
In this paper, we shall give a satisfactory answer in dimension 1 for generally hyperbolic systems of balance laws.
Before stating main results, let us introduce some notations and review structural conditions, which are designed for the global-in-time and decay stability of classical solutions. Set
\begin{eqnarray*}
\mathcal{M}=\{\psi\in\mathbb{R}^{N}: \langle\psi,G(U)\rangle=0\ \
\mbox{for any}\ U\in \mathcal{O}_{U}\}.
\end{eqnarray*}
Then $\mathcal{M}$ is a subset of $\mathbb{R}^{N}$ with $\mathrm{dim}\
\mathcal{M}=N_{1}$. From the definition of $\mathcal{M}$, we have
\begin{eqnarray*}
G(U)\in\mathcal{M}^{\top}(\mbox{the orthogonal complement of}\
\mathcal{M}),\ \mbox{for any}\ U\in \mathcal{O}_{U}.
\end{eqnarray*}
Furthermore, corresponding to the orthogonal decomposition
$\mathbb{R}^{N}=\mathcal{M}\oplus\mathcal{M}^{\top}$, we may write
$U\in\mathbb{R}^{N}$ as
\begin{eqnarray*}
U=\left(
    \begin{array}{c}
      U_{1} \\
      U_{2} \\
    \end{array}
  \right)
\end{eqnarray*}
such that $U\in\mathcal{M}$ holds if and only if $U_{2}=0$. We
denote by $\mathcal{E}$ the set of equilibrium state for the balance
laws (\ref{R-E1}):
\begin{eqnarray*}
\mathcal{E}=\{U\in \mathcal{O}_{U}: G(U)=0\}.
\end{eqnarray*}
For $\bar{U}\in \mathcal{E}$, we define
$U=\bar{U}+D_{W}U(\bar{W})z$, where $W$ is the symmetrized variable related to the entropy function endowed
and $\bar{W}\in \mathcal{M}$ is the corresponding equilibrium of $\bar{U}$, see \cite{KY2} for details.
Then (\ref{R-E1}) can be rewritten as
\begin{equation}
A^{0}z_{t}+\sum_{k=1}^{n}A^{k}z_{x_{j}}+Lz=\sum_{k=1}^{n}P^{k}_{x_{k}}+Q, \label{R-E3}
\end{equation}
where
$A^{0}=A^{0}(\bar{W}),
A^{k}=A^{k}(\bar{W})$ and $L=L(\bar{W})$ are constant
matrices, and
$$P^{k}=-[F^{k}(U)-F^{k}(\bar{U})-D_{U}F^{k}(\bar{U})(U-\bar{U})],$$
$$Q=G(U)-G(\bar{U})-D_{U}G(\bar{U})(U-\bar{U}).$$
Notice that $P:=(P^{1},P^{2},\cdot\cdot\cdot,P^{n})=O(z^2)$, $ Q=O(z^2)\in \mathcal{M}^{\bot}$ satisfying $P(0)=P'(0)=0$ and $Q(0)=Q'(0)=0$.
The initial datum is given correspondingly by
\begin{equation}
z|_{t=0}=z_{0}. \label{R-E4}
\end{equation}

Denote
\begin{eqnarray}
\widehat{\mathcal{G}f}(t,\xi)=e^{t\Phi(i\xi)}\hat{f}(\xi), \label{R-E5}
\end{eqnarray}
where $$\Phi(i\xi)=-(A^{0})^{-1}[A(i\xi)+L]$$ with
$A(i\xi)=i\sum_{k=1}^{n}A^{k}\xi_{k}$. Then $\mathcal{G}(t,x)z_{0}$ is the solution of linearized system
\begin{equation}
\left\{
\begin{array}{l}
A^{0}z_{t}+\sum_{k=1}^{n}A^{k}z_{x_{k}}+Lz=0,\\
z|_{t=0}=z_{0}.
\end{array}\right. \label{R-E6}
\end{equation}

Next, we state structural assumptions in \cite{SK,UKS} addressed by the second author and his collaborators. The first assumption ensures the system
(\ref{R-E1}) is ``symmetric hyperbolic" in the sense that

\noindent \textbf{Condition (A):} the matrix $A^{0}$ is real symmetric and positive definite, $A^{k}(k=1,2,\cdot\cdot\cdot,n)$ are real symmetric. $L$ is real symmetric and nonnegative definite, whose kernel coincides with $\mathcal{M}$.

Further assume the following condition which is referred as ``Kawashima-Shizuta" algebraic condition.

\noindent \textbf{Condition (K):} There is a real compensating matrix $K(\omega)\in C^{\infty}(\mathbb{S}^{n-1})$ with the following properties: $K(-\omega)=-K(\omega), (K(\omega)A^{0})^{\top}=-K(\omega)A^{0}$ and
\begin{eqnarray}
(K(\omega)A(\omega))_{1}>0   \ \ \ \mbox{on}\ \mathcal{M} \label{R-E7}
\end{eqnarray}
for each $\omega\in\mathbb{S}^{n-1}$, where $(\cdot)_{1}$ denotes the symmetric part of a matrix.

As shown by \cite{SK,UKS}, the dissipative structure of (\ref{R-E6}) holds by assuming the Condition (A) and Condition (K):
$${\rm Re}\,\lambda(i\xi)\leq -c\eta_1(\xi), \quad \mbox{with}\quad  \eta_1(\xi)=\frac{|\xi|^2}{1+|\xi|^2},$$
where $\lambda(i\xi)$ is the character root of (\ref{R-E6}) in Fourier spaces. Recently, we develop the following decay properties
by performing the L-P pointwise energy estimates and the interpolation inequality in Lemma \ref{lem6.2}.

\begin{prop}\label{prop1.1} (\cite{XK2})
If $z_{0}\in \dot{B}^{\sigma}_{2,1}(\mathbb{R}^{n})\cap \dot{B}^{-s}_{2,\infty}(\mathbb{R}^{n})$ for $\sigma\geq0$ and $s>0$, then the solutions $z(t,x)$ of (\ref{R-E6}) has the decay estimate
\begin{equation}
\|\Lambda^{\ell}z\|_{B_{2,1}^{\sigma-\ell}}\lesssim \|z_{0}\|_{\dot{B}^{-s}_{2,\infty}}(1+t)^{-\frac{\ell+s}{2}}+\|z_{0}\|_{\dot{B}_{2,1}^{\sigma}}e^{-ct} \label{R-E8}
\end{equation}
for $0\leq\ell\leq\sigma$.
\end{prop}
Additionally, we have a analogue on the framework of homogeneous Besov spaces.
\begin{prop} (\cite{XK2}) \label{prop1.2}
If $z_{0}\in \dot{B}^{\sigma}_{2,1}(\mathbb{R}^{n})\cap \dot{B}^{-s}_{2,\infty}(\mathbb{R}^{n})$ for $\sigma\in \mathbb{R}, s\in \mathbb{R}$ satisfying $\sigma+s>0$, then the solution $z(t,x)$ of (\ref{R-E6}) has the decay estimate
\begin{equation}
\|z\|_{\dot{B}_{2,1}^{\sigma}}\lesssim \|z_{0}\|_{\dot{B}^{-s}_{2,\infty}}(1+t)^{-\frac{\sigma+s}{2}}+\|z_{0}\|_{\dot{B}_{2,1}^{\sigma}}e^{-ct}. \label{R-E9}
\end{equation}
\end{prop}

\subsection{Main results}
Based on Propositions \ref{prop1.1}-\ref{prop1.2}, the optimal decay rates for the nonlinear system have been shown in \cite{XK2} by
localized time-weighted energy approaches, with the aid of the frequency-localization Duhamel principle and improved Gagliardo-Nirenberg-Sobolev inequality.
However, decay results hold for higher dimensions ($n\geq3$) due to interpolation techniques. In the dimension 1, indeed, we met with an almost insurmountable obstacle at the low-frequency. To overcome the difficulty, in the current paper, we develop new time-decay properties for frequency-localization Duhamel principle, which lead to
more elaborate low-frequency analysis. Precisely, the low-frequency estimate is divided into two parts, and on each part, different values (for example, $\sigma=0$ or $\sigma=s+1/2)$ of the derivation index $\sigma$ in Lemmas \ref{lem3.3}-\ref{lem3.4} can be chosen, see (\ref{R-E39}), (\ref{R-E41}) and (\ref{R-E47}) for details. On the other hand, we involve a new observation: the advantage of (\ref{R-E3}) rather than the standard normal form adopted in \cite{XK2} lies in the nice nonlinear structure. Observe that the composite functions $P$ and $Q$ both have the form of $O(z^2)$,
which is very helpful to obtain desired decay estimates, see (\ref{R-E42})
and (\ref{R-E48}).

To show main results more explicitly, we first define the functional spaces
$$\widetilde{\mathcal{C}}_{T}(B^{s}_{p,r}):=\widetilde{L}^{\infty}_{T}(B^{s}_{p,r})\cap\mathcal{C}([0,T],B^{s}_{p,r})
$$ and $$\widetilde{\mathcal{C}}^1_{T}(B^{s}_{p,r}):=\{f\in\mathcal{C}^1([0,T],B^{s}_{p,r})|\partial_{t}f\in\widetilde{L}^{\infty}_{T}(B^{s}_{p,r})\},$$
where the index $T$ will be omitted when $T=+\infty$. Among them, Chemin-Lerner
spaces $\widetilde{L}^{\theta}_{T}(B^{s}_{p,r})$ was first initialled by J.-Y. Chemin and N. Lerner in \cite{CL}. The interested reader is also referred to  \cite{BCD} for their definitions.

Now, as a counterpart of \cite{XK2}, we state the decay estimate for the case $n=1$.
\begin{thm}\label{thm1.1}
Let $z(t,x)$ be the global classical solution in the space $\widetilde{\mathcal{C}}(B^{3/2}_{2,1}(\mathbb{R}))\\ \cap
\widetilde{\mathcal{C}}^1(B^{1/2}_{2,1}(\mathbb{R}))$ given by \cite{XK1}. Suppose that $z_{0}\in B^{3/2}_{2,1}(\mathbb{R})\cap \dot{B}^{-s}_{2,\infty}(\mathbb{R})(1/4<s\leq 1/2)$ and the norm
$E_{0}:=\|z_{0}\|_{B^{3/2}_{2,1}(\mathbb{R})\cap \dot{B}^{-s}_{2,\infty}(\mathbb{R})}$ is sufficiently small. Then it holds that
\begin{eqnarray}
\|z(t)\|_{X_{0}}\lesssim E_{0}(1+t)^{-\frac{s+\ell}{2}}\label{R-E10}
\end{eqnarray}
where $X_{0}=:\dot{B}^{\ell}_{2,1}(\mathbb{R})$ if $0<\ell\leq 1/2$, and $X_{0}=:L^2(\mathbb{R})$ if $\ell=0$. In particular,
owing to the embedding $\dot{B}^{0}_{2,1}(\mathbb{R})\hookrightarrow L^2(\mathbb{R})$, one further has
\begin{eqnarray}
\|\Lambda^{\ell}z(t)\|_{L^2(\mathbb{R})}\lesssim E_{0}(1+t)^{-\frac{s+\ell}{2}}\label{R-E11}
\end{eqnarray}
for $0\leq\ell\leq 1/2$.
\end{thm}

\begin{rem}
The proof of Theorem \ref{thm1.1} depends on localized time-weighted energy argument. Here, the new energy functionals not only contain different time-weighted norms according to the derivative index, but allow to the derivative case of fractional order as well, which can be viewed as the improvement
of classical time-weighted functionals in \cite{Ma}. The crucial ingredient of our argument lies in
new time-decay properties for the frequency-localization Duhamel principle. For instance, Lemma \ref{lem3.2} (inhomogeneous version) and Lemmas \ref{lem3.3} (homogeneous version) are used to take care of the variable of derivative form in (\ref{R-E3}), whereas Lemma \ref{lem3.4} (homogeneous version) is responsible for non-degenerate dissipative variables.
\end{rem}

\begin{rem}
Due to the case of one dimension, let us mention that  $s\in (1/4,1/2]$ in Theorem \ref{thm1.1} is needed to be restricted additionally in comparison with those results in higher dimensions (see \cite{XK2}).
\end{rem}

Note that the $L^p(\mathbb{R}^{n})$ embedding in Lemma \ref{lem6.5}, the optimal decay rates of $L^1(\mathbb{R})$-$L^2(\mathbb{R})$ are available.
\begin{thm}\label{thm1.2}
Let $z(t,x)$ be the global classical solution in $\widetilde{\mathcal{C}}(B^{3/2}_{2,1}(\mathbb{R}))\cap
\widetilde{\mathcal{C}}^1(B^{1/2}_{2,1}(\mathbb{R}))$ given by \cite{XK1}. If further the initial data $z_{0}\in L^1(\mathbb{R})$ and
$$\widetilde{E}_{0}:=\|z_{0}\|_{B^{3/2}_{2,1}(\mathbb{R})\cap L^1(\mathbb{R})}$$
is sufficiently small. Then the classical solutions $z(t,x)$  satisfies the optimal decay estimate
\begin{eqnarray}
\|\Lambda^{\ell}z(t,\cdot)\|_{L^2(\mathbb{R})}\lesssim \widetilde{E}_{0}(1+t)^{-\frac{1}{4}-\frac{\ell}{2}} \label{R-E12}
\end{eqnarray}
for $0\leq\ell\leq 1/2$.
\end{thm}

\begin{rem}
Theorem \ref{thm1.2} is just the direct consequence of Theorem \ref{thm1.1}.
The harmonic analysis allows to reduce significantly the regularity requirements
on the initial data in comparison with \cite{BHN,KY2}. In addition, it is shown that
the classical solution and its derivatives of fractional order decay at the optimal algebraic rate
in the whole interval $[0,1/2]$.
\end{rem}

The paper is organized as follows. In Sect.~\ref{sec:2}, we recall the Littlewood-Paley decomposition theory and Besov spaces.
Some useful properties in Besov spaces are also presented. Sect.~\ref{sec:3} is devoted to develop new time-decay properties for the frequency-localization Duhamel principle. In Sect.~\ref{sec:4}, we perform time-weighted energy approaches in terms of the low-frequency and high-frequency decompositions along with interpolation inequalities to deduce the desired decay estimates. As applications, in Sect.~\ref{sec:5}, we prove decay results for
concrete hyperbolic systems, including damped compressible Euler equations, Thermoelasticity with second sound and Timoshenko systems with equal speeds.
In Sect.~\ref{sec:6} (Appendix), for the convenience of readers, some interpolation inequalities related to Besov spaces, for instance, $L^{p}(\mathbb{R}^{n})$ embedding and improved Gagliardo-Nirenberg-Sobolev inequality, will be presented.

\textbf{Notations}. Throughout the paper,  we use
$\langle\cdot,\cdot\rangle$ to denote the standard inner product in
the real $\mathbb{R}^{n}$ or complex $\mathbb{C}^{n}$.
$f\lesssim g$ denotes $f\leq Cg$, where $C>0$
is a generic constant. $f\thickapprox g$ means $f\lesssim g$ and
$g\lesssim f$.  Denote by $\mathcal{C}([0,T],X)$ (resp., $\mathcal{C}^{1}([0,T],X)$)
the space of continuous (resp., continuously differentiable)
functions on $[0,T]$ with values in a Banach space $X$.

\section{Littlewood-Paley theory and Besov spaces}\label{sec:2}
The proofs of most of the results presented  require a
dyadic decomposition of Fourier variables, so we recall briefly the
Littlewood-Paley decomposition and Besov spaces in $\mathbb{R}^{n}(n\geq1)$, the reader is referred to
\cite{BCD} for details.

Let ($\varphi, \chi)$ is a
couple of smooth functions valued in [0,1] such that $\varphi$ is
supported in the shell
$\mathcal{C}(0,\frac{3}{4},\frac{8}{3})=\{\xi\in\mathbb{R}^{n}|\frac{3}{4}\leq|\xi|\leq\frac{8}{3}\}$,
$\chi$ is supported in the ball $\mathcal{B}(0,\frac{4}{3})=
\{\xi\in\mathbb{R}^{n}||\xi|\leq\frac{4}{3}\}$ satisfying
$$
\chi(\xi)+\sum_{j\geq0}\varphi(2^{-j}\xi)=1,\ \ \ \ j\in \mathbb{N}\cup\{0\},\ \ \xi\in\mathbb{R}^{n}
$$
and
$$
\sum_{j\in\mathbb{Z}}\varphi(2^{-j}\xi)=1,\ \ \ \ j\in \mathbb{Z},\
\ \xi\in\mathbb{R}^{n}\setminus\{0\}.
$$
For $f\in\mathcal{S'}$(the set of temperate distributions
which is the dual of the Schwarz class $\mathcal{S}$),  define
$$
\Delta_{-1}f:=\chi(D)f=\mathcal{F}^{-1}(\chi(\xi)\mathcal{F}f),\
\Delta_{j}f:=0 \ \  \mbox{for}\ \  j\leq-2;
$$
$$
\Delta_{j}f:=\varphi(2^{-j}D)f=\mathcal{F}^{-1}(\varphi(2^{-j}\xi)\mathcal{F}f)\
\  \mbox{for}\ \  j\geq0;
$$
$$
\dot{\Delta}_{j}f:=\varphi(2^{-j}D)f=\mathcal{F}^{-1}(\varphi(2^{-j}\xi)\mathcal{F}f)\
\  \mbox{for}\ \  j\in\mathbb{Z},
$$
where $\mathcal{F}f$, $\mathcal{F}^{-1}f$ represent the Fourier
transform and the inverse Fourier transform on $f$, respectively. Observe that
the operator $\dot{\Delta}_{j}$ coincides with $\Delta_{j}$ for $j\geq0$.

Denote by $\mathcal{S}'_{0}:=\mathcal{S}'/\mathcal{P}$ the tempered
distributions modulo polynomials $\mathcal{P}$. We first give the definition of homogeneous Besov spaces.

\begin{defn}\label{defn2.1}
For $s\in \mathbb{R}$ and $1\leq p,r\leq\infty,$ the homogeneous
Besov spaces $\dot{B}^{s}_{p,r}$ is defined by
$$\dot{B}^{s}_{p,r}=\{f\in \mathcal{S}'_{0}:\|f\|_{\dot{B}^{s}_{p,r}}<\infty\},$$
where
$$\|f\|_{\dot{B}^{s}_{p,r}}
=\left\{\begin{array}{l}\Big(\sum_{j\in\mathbb{Z}}(2^{js}\|\dot{\Delta}_{j}f\|_{L^p})^{r}\Big)^{1/r},\
\ r<\infty, \\ \sup_{j\in\mathbb{Z}}
2^{js}\|\dot{\Delta}_{j}f\|_{L^p},\ \ r=\infty.\end{array}\right.
$$\end{defn}

Similarly, the definition of inhomogeneous Besov spaces is stated as follows.
\begin{defn}\label{defn2.2}
For $s\in \mathbb{R}$ and $1\leq p,r\leq\infty,$ the inhomogeneous
Besov spaces $B^{s}_{p,r}$ is defined by
$$B^{s}_{p,r}=\{f\in \mathcal{S}':\|f\|_{B^{s}_{p,r}}<\infty\},$$
where
$$\|f\|_{B^{s}_{p,r}}
=\left\{\begin{array}{l}\Big(\sum_{j=-1}^{\infty}(2^{js}\|\Delta_{j}f\|_{L^p})^{r}\Big)^{1/r},\
\ r<\infty, \\ \sup_{j\geq-1} 2^{js}\|\Delta_{j}f\|_{L^p},\ \
r=\infty.\end{array}\right.
$$
\end{defn}

we present some useful facts as follows. The first one is the improved
Bernstein inequality, see, e.g., \cite{W}.

\begin{lem}\label{lem2.1}
Let $0<R_{1}<R_{2}$ and $1\leq a\leq b\leq\infty$.
\begin{itemize}
\item [(i)] If $\mathrm{Supp}\mathcal{F}f\subset \{\xi\in \mathbb{R}^{n}: |\xi|\leq
R_{1}\lambda\}$, then
\begin{eqnarray*}
\|\Lambda^{\alpha}f\|_{L^{b}}
\lesssim \lambda^{\alpha+n(\frac{1}{a}-\frac{1}{b})}\|f\|_{L^{a}}, \ \  \mbox{for any}\ \  \alpha\geq0;
\end{eqnarray*}

\item [(ii)]If $\mathrm{Supp}\mathcal{F}f\subset \{\xi\in \mathbb{R}^{n}:
R_{1}\lambda\leq|\xi|\leq R_{2}\lambda\}$, then
\begin{eqnarray*}
\|\Lambda^{\alpha}f\|_{L^{a}}\approx\lambda^{\alpha}\|f\|_{L^{a}}, \ \  \mbox{for any}\ \ \alpha\in\mathbb{R}.
\end{eqnarray*}
\end{itemize}
\end{lem}
As a consequence of the above inequality, we have
$$\|\Lambda^{\alpha} f\|_{B^s_{p, r}}\lesssim\|f\|_{B^{s +\alpha}_{p,
r}} \ (\alpha\geq0); \ \ \ \|\Lambda^{\alpha} f\|_{\dot{B}^s_{p, r}}\approx \|f\|_{\dot{B}^{s+\alpha}_{p, r}} \ (\alpha\in\mathbb{R}).$$

Besov spaces obey various inclusion relations. Precisely,
\begin{lem}\label{lem2.2} Let $s\in \mathbb{R}$ and $1\leq
p,r\leq\infty,$ then
\begin{itemize}
\item[(1)]If $s>0$, then $B^{s}_{p,r}=L^{p}\cap \dot{B}^{s}_{p,r};$
\item[(2)]If $\tilde{s}\leq s$, then $B^{s}_{p,r}\hookrightarrow
B^{\tilde{s}}_{p,r}$. This inclusion relation is false for
the homogeneous Besov spaces;
\item[(3)]If $1\leq r\leq\tilde{r}\leq\infty$, then $\dot{B}^{s}_{p,r}\hookrightarrow
\dot{B}^{s}_{p,\tilde{r}}$ and $B^{s}_{p,r}\hookrightarrow
B^{s}_{p,\tilde{r}};$
\item[(4)]If $1\leq p\leq\tilde{p}\leq\infty$, then $\dot{B}^{s}_{p,r}\hookrightarrow \dot{B}^{s-n(\frac{1}{p}-\frac{1}{\tilde{p}})}_{\tilde{p},r}
$ and $B^{s}_{p,r}\hookrightarrow
B^{s-n(\frac{1}{p}-\frac{1}{\tilde{p}})}_{\tilde{p},r}$;
\item[(5)]$\dot{B}^{n/p}_{p,1}\hookrightarrow\mathcal{C}_{0},\ \ B^{n/p}_{p,1}\hookrightarrow\mathcal{C}_{0}(1\leq p<\infty);$
\end{itemize}
where $\mathcal{C}_{0}$ is the space of continuous bounded functions
which decay at infinity.
\end{lem}

Below are the Moser-type product estimates, which plays an important role in the estimate of bilinear
terms.
\begin{prop}\label{prop2.1}
Let $s>0$ and $1\leq
p,r\leq\infty$. Then $\dot{B}^{s}_{p,r}\cap L^{\infty}$ is an algebra and
$$
\|fg\|_{\dot{B}^{s}_{p,r}}\lesssim \|f\|_{L^{\infty}}\|g\|_{\dot{B}^{s}_{p,r}}+\|g\|_{L^{\infty}}\|f\|_{\dot{B}^{s}_{p,r}}.
$$
Let $s_{1},s_{2}\leq n/p$ such that $s_{1}+s_{2}>n\max\{0,\frac{2}{p}-1\}. $  Then one has
$$\|fg\|_{\dot{B}^{s_{1}+s_{2}-n/p}_{p,1}}\lesssim \|f\|_{\dot{B}^{s_{1}}_{p,1}}\|g\|_{\dot{B}^{s_{2}}_{p,1}}.$$
\end{prop}

Finally, we state continuity results for compositions to end this section.

\begin{prop}\label{prop2.2}
Let $s>0$, $1\leq p, r\leq \infty$ and $F'\in
W^{[s]+1,\infty}_{loc}(I;\mathbb{R})$ with $F(0)=0$. Assume that $f\in \dot{B}^{s}_{p,r}\cap
L^{\infty}$, then there exists a function $\mathbf{C}$ depending only on $s,p,r,n,$ and $F$ such that
$$\|F(f)\|_{\dot{B}^{s}_{p,r}}\leq
\mathbf{C}(\|f\|_{L^{\infty}})\|f\|_{\dot{B}^{s}_{p,r}}.$$
\end{prop}

\begin{prop}\label{prop2.3}
Let $s>0$, $1\leq p, r\leq \infty$ and $F\in
W^{[s]+3,\infty}_{loc}(I;\mathbb{R})$ with $F(0)=0$.  Assume that $f\in \dot{B}^{s}_{p,r}\cap
L^{\infty},$ then there exists a function $\mathbf{C}$ depending only on $s,p,r,n,$ and $F$ such that
$$
\|F(f)-F'(0)f\|_{\dot{B}^{s}_{p,r}}\leq
\mathbf{C}(\|f\|_{L^{\infty}})\|f\|^2_{\dot{B}^{s}_{p,r}}.
$$
\end{prop}

\section{Frequency-localization Duhamel principle} \setcounter{equation}{0} \label{sec:3}
In this section, we develop time-decay properties for the frequency-localization Duhamel principle, which allow to
perform more elaborate analysis at the low frequency in comparison with the recent work \cite{XK2}. Without loss of generality,
these results are adapted to arbitrary dimensional spaces except for Lemma \ref{lem3.4}.
First of all, we present
a frequency-localization Duhamel principle for the nonlinear system (\ref{R-E3})-(\ref{R-E4}).
\begin{lem}\label{lem3.1}
Suppose that $z(t,x)$ is a solution of (\ref{R-E3})-(\ref{R-E4}). Then it holds that
\begin{eqnarray}
\Delta_{j}\Lambda^{\ell}z(t,x)&=&\Delta_{j}\Lambda^{\ell}[\mathcal{G}(t,x)z_{0}]
\nonumber\\&&+\int^{t}_{0}\Delta_{j}\Lambda^{\ell}\Big[\mathcal{G}(t-\tau,x)(A^{0})^{-1}\Big(\sum^{n}_{k=1}P^{k}_{x_{k}}+Q\Big)\Big]d\tau \label{R-E13}
\end{eqnarray}
for $j\geq-1$ and $\ell\in \mathbb{R}$, and
\begin{eqnarray}
\dot{\Delta}_{j}\Lambda^{\ell}z(t,x)&=&\dot{\Delta}_{j}\Lambda^{\ell}[\mathcal{G}(t,x)z_{0}]
\nonumber\\&&+\int^{t}_{0}\dot{\Delta}_{j}\Lambda^{\ell}\Big[\mathcal{G}(t-\tau,x)(A^{0})^{-1}\Big(\sum^{n}_{k=1}P^{k}_{x_{k}}+Q\Big)\Big]d\tau\label{R-E14}
\end{eqnarray}
for $j\in\mathbb{Z}$ and $\ell\in \mathbb{R}$.
\end{lem}

\begin{proof}
The proof follows from the standard Duhamel principle and the definition of $\mathcal{G}(t)$ in (\ref{R-E5}). See \cite{XK2} for similar details.
\end{proof}

Secondly, we prove the decay property for the variable of derivative form at the low-frequency.

\begin{lem}\label{lem3.2} For $\ell+1-\sigma\geq0$ and $s>0$, it holds that
\begin{eqnarray}
\Big\|\Delta_{-1}\Lambda^{\ell}\Big[\mathcal{G}(t)\Big((A^{0})^{-1}\sum^{n}_{k=1}P^{k}_{x_{k}}\Big)\Big]\Big\|_{L^2}\lesssim (1+t)^{-\frac{s+\ell+1-\sigma}{2}}\|\Lambda^{\sigma}P\|_{\dot{B}^{-s}_{2,\infty}}, \label{R-E15}
\end{eqnarray}
where $P=(P^{1},P^{2}\cdot\cdot\cdot,P^{n})$.
\end{lem}

\begin{proof}
It follows from Plancherel's theorem that
\begin{eqnarray}
\Big\|\Delta_{-1}\Lambda^{\ell}\Big[\mathcal{G}(t)\Big((A^{0})^{-1}\sum^{n}_{k=1}P^{k}_{x_{k}}\Big)\Big]\Big\|_{L^2}\lesssim
\|\Delta_{-1}\Lambda^{\ell'}\mathcal{G}(t)(\Lambda^{\sigma}P)\|_{L^2} \label{R-E16}
\end{eqnarray}
with $\ell'=\ell+1-\sigma$ for $\ell'\geq0$. From the definition of $\mathcal{G}(t)$, we know that $\tilde{z}=\mathcal{G}(t)(\Lambda^{\sigma}P)$ is the solution of (\ref{R-E6}) with the initial data $\Lambda^{\sigma}P$. According to L-P pointwise energy estimates in \cite{XK2}, we achieve that
\begin{eqnarray}
\frac{d}{dt}E[\widehat{\tilde{z}_{-1}}]+c|\xi|^2|\widehat{\tilde{z}_{-1}}|^2\leq0, \label{R-E17}
\end{eqnarray}
for $c>0$, where $\tilde{z}_{-1}:=\Delta_{-1}\tilde{z}.$

Multiplying (\ref{R-E17}) with $|\xi|^{2\ell'}$ and integrating the resulting inequality  over $\mathbb{R}^{n}_{\xi}$, with the aid of Plancherel's theorem, we arrive at
\begin{eqnarray}
\frac{d}{dt}\mathcal{E}[\tilde{z}_{-1}]^2+c_{3}\|\Lambda^{\ell'+1}\tilde{z}_{-1}\|^2_{L^2}\leq0, \label{R-E18}
\end{eqnarray}
where $$\mathcal{E}[\tilde{z}_{-1}]:=\Big(\int_{\mathbb{R}^{n}_{\xi}}|\xi|^{2\ell'}E[\widehat{\tilde{z}_{-1}}]d\xi\Big)^{1/2}\approx\|\Lambda^{\ell'}\tilde{z}_{-1}\|_{L^2}.$$
According to the interpolation inequality related the Besov space $\dot{B}^{-s}_{2,\infty}$ (see Lemma \ref{lem6.2}), we arrive at
\begin{eqnarray}
\|\Lambda^{\ell'}\tilde{z}_{-1}\|_{L^2} &\lesssim&\|\Lambda^{\ell'+1}\tilde{z}_{-1}\|^{\theta}_{L^2}\|\tilde{z}_{-1}\|^{1-\theta}_{\dot{B}^{-s}_{2,\infty}}\ \ \  \Big(\theta=\frac{\ell'+s}{\ell'+1+s}\Big)\nonumber\\ &\lesssim& \|\Lambda^{\ell'+1}\tilde{z}_{-1}\|^{\theta}_{L^2}\|\tilde{z}\|^{1-\theta}_{\dot{B}^{-s}_{2,\infty}}, \label{R-E19}
\end{eqnarray}

In addition, by applying the homogeneous operator $\dot{\Delta}_{j}(j\in \mathbb{Z})$ to the system (\ref{R-E6}) and performing the inter product with
$\dot{\Delta}_{j}\tilde{z}$, we can infer that there exists a constant $c_{0}>0$ such that
\begin{eqnarray}
\frac{1}{2}\frac{d}{dt}(\tilde{A}^{0}\widehat{\dot{\Delta}_{j}\tilde{z}},\widehat{\dot{\Delta}_{j}\tilde{z}})+c_{0}\|(I-\mathcal{P})\dot{\Delta}_{j}\tilde{z}\|^2_{L^2}\leq0, \label{R-E20}
\end{eqnarray}
where $\mathcal{P}$ is the orthogonal projection onto $\mathcal{M}=\mathrm{ker}L$. This immediately leads to
\begin{eqnarray}
\|\tilde{z}\|_{\dot{B}^{-s}_{2,\infty}}\leq\|\Lambda^{\sigma}P\|_{\dot{B}^{-s}_{2,\infty}}.  \label{R-E21}
\end{eqnarray}
Together with (\ref{R-E20}) and (\ref{R-E21}), we are led to the differential inequality
\begin{eqnarray}
\frac{d}{dt}\mathcal{E}[\tilde{z}_{-1}]^2+C\|\Lambda^{\sigma}P\|_{\dot{B}^{-s}_{2,\infty}}^{-\frac{2}{s}}(\|\Lambda^{\ell'}\tilde{w}_{-1}\|^2_{L^2})^{1+\frac{1}{\ell'+s}}\leq0,
\label{R-E22}
\end{eqnarray}
which yields
\begin{eqnarray}
\|\Lambda^{\ell'}\tilde{z}_{-1}\|_{L^2}\lesssim\|\Lambda^{\sigma}P\|_{\dot{B}^{-s}_{2,\infty}}(1+t)^{-\frac{\ell'+s}{2}}.\label{R-E23}
\end{eqnarray}
Note that (\ref{R-E16}), the inequality (\ref{R-E23}) leads to (\ref{R-E15}) directly.
\end{proof}

Similarly, for homogeneous decompositions, we have an analogue of decay property at the low-frequency.
\begin{lem}\label{lem3.3} For $s+\ell+1-\sigma>0$, it holds that
\begin{eqnarray}
&&\Big\|2^{j\ell}\Big\|\dot{\Delta}_{j}\Big[\mathcal{G}(t)\Big((A^{0})^{-1}\sum^{n}_{k=1}P^{k}_{x_{k}}\Big)\Big]\Big\|_{L^2}\Big\|_{l^{r}_{j}(j\leq j_{0})} \nonumber\\ &\lesssim& (1+t)^{-\frac{s+\ell+1-\sigma}{2}}\|\Lambda^{\sigma}P\|_{\dot{B}^{-s}_{2,\infty}} \label{R-E24}
\end{eqnarray}
where $P=(P^{1},P^{2}\cdot\cdot\cdot,P^{n})$, $r\in [1,+\infty]$ and
$j_{0}\in \mathbb{Z}$ is to be determined in Lemma \ref{lem3.4} below.
\end{lem}
\begin{proof}
It follows from that
\begin{eqnarray}
\Big\|\dot{\Delta}_{j}\Big[\mathcal{G}(t)\Big((A^{0})^{-1}\sum^{n}_{i=1}P^{k}_{x_{k}}\Big)\Big]\Big\|_{L^2}&\lesssim& e^{-c|\xi|^2t}\|\dot{\Delta}_{j}P^{k}_{x_{k}}\|_{L^2}
\nonumber\\&\lesssim& 2^{j(1-\sigma)}e^{-c2^{2j}t}\|\dot{\Delta}_{j}\Lambda^{\sigma}p\|_{L^2} \label{R-E25}
\end{eqnarray}
which implies that
\begin{eqnarray}
&&2^{j\ell}\Big\|\dot{\Delta}_{j}\Big[\mathcal{G}(t)\Big((A^{0})^{-1}\sum^{n}_{k=1}P^{k}_{x_{k}}\Big)\Big]\Big\|_{L^2}\nonumber\\&\lesssim&
(1+t)^{-\frac{s+\ell+1-\sigma}{2}}\|\Lambda^{\sigma}P\|_{\dot{B}^{-s}_{2,\infty}}
\Big[(2^{j}\sqrt{t})^{(s+\ell+1-\sigma)}e^{-c(2^{j}\sqrt{t})^2}\Big]. \label{R-E26}
\end{eqnarray}
Note that
\begin{eqnarray}
\Big\|(2^{j}\sqrt{t})^{(s+\ell+1-\sigma)}e^{-c(2^{j}\sqrt{t})^2}\Big\|_{l^{r}_{j}}\lesssim 1, \label{R-E27}
\end{eqnarray}
for any $r\in [1,+\infty]$, we deduce that
\begin{eqnarray}
&&\Big\|2^{j\ell}\Big\|\dot{\Delta}_{j}\Big[\mathcal{G}(t)\Big((A^{0})^{-1}\sum^{n}_{k=1}P^{k}_{x_{k}}\Big)\Big]\Big\|_{L^2}\Big\|_{l^{r}_{j}(j\leq j_{0})} \nonumber\\ &\lesssim& (1+t)^{-\frac{s+\ell+1-\sigma}{2}}\|\Lambda^{\sigma}P\|_{\dot{B}^{-s}_{2,\infty}} \label{R-E28}
\end{eqnarray}
for $s+\ell+1-\sigma>0$, which is the inequality (\ref{R-E24}) exactly.
\end{proof}

For the non-degenerate variable $Q\in\mathcal{M}^{\top}$, we have a sharper decay property at the low-frequency in the case of $n=1$.
\begin{lem}\label{lem3.4}For $n=1$, if $Q\in \mathcal{M}^{\top}$, then it holds that
\begin{eqnarray}
&&\Big\|2^{j\ell}\|\dot{\Delta}_{j}[\mathcal{G}(t)(A^{0})^{-1}Q]\|_{L^2}\Big\|_{l^{r}_{j}(j\leq j_{0})} \nonumber\\ &\lesssim & e^{-ct}\|2^{j\ell}\|\dot{\Delta}_{j}Q\|_{L^2}\|_{l^{r}_{j}(j\leq j_{0})}+(1+t)^{-\frac{s+\ell+1-\sigma}{2}}\|\Lambda^{\sigma}Q\|_{\dot{B}^{-s}_{2,\infty}} \label{R-E29}
\end{eqnarray}
for $s+\ell+1-\sigma>0$ and $r\in[1,+\infty]$.
\end{lem}

\begin{proof}
To prove this, we recall a conclusion achieved in \cite{KY2} for the case of one dimension. Precisely, there exists a constant $r_{0}>0$ such that
\begin{eqnarray}
|e^{t\Phi(i\xi)}(A^{0})^{-1}\hat{Q}(\xi)|\lesssim e^{-ct}|\hat{Q}(\xi)|+|\xi|e^{-c|\xi|^2t}|\hat{Q}(\xi)| \label{R-E30}
\end{eqnarray}
for $|\xi|\leq r_{0}.$  From (\ref{R-E30}), we take the Fourier cut-off of $\dot{\Delta}_{j}$ with $j\leq j_{0}:=[\log_{2}(\frac{3}{8}r_{0})]$ \footnote{$[x]$ means the maximum integer value which is less than $x$.} to get
\begin{eqnarray}
|e^{t\Phi(i\xi)}(A^{0})^{-1}\widehat{\dot{\Delta}_{j}Q}|\lesssim e^{-ct}|\widehat{\dot{\Delta}_{j}Q}|+|\xi|e^{-c|\xi|^2t}|\widehat{\dot{\Delta}_{j}Q}|. \label{R-E31}
\end{eqnarray}
It follows from Plancherel's theorem that
\begin{eqnarray}
&&\|\dot{\Delta}_{j}[\mathcal{G}(t)(A^{0})^{-1}Q]\|_{L^2}\nonumber\\& \lesssim& e^{-ct}\|\dot{\Delta}_{j}Q\|_{L^2}+2^{j(1-\sigma)}e^{-c2^{2j}t}\|\dot{\Delta}_{j}\Lambda^{\sigma}Q\|_{L^2}. \label{R-E32}
\end{eqnarray}
Then, similar to the procedure leading to (\ref{R-E28}), we arrive at (\ref{R-E29}) readily.
\end{proof}

\section{Localized time-weighted energy approaches} \setcounter{equation}{0} \label{sec:4}
Based on decay properties developed in Section \ref{sec:3}, the aim of this section is to derive decay estimates for (\ref{R-E3})-(\ref{R-E4}).
For that purpose,  time-weighted energy approaches in terms of low-frequency and high-frequency decompositions, as well as
improved Gagliardo-Nirenberg-Sobolev inequalities are mainly performed. In this section, our discussion is restricted to the case $n=1$.
Firstly, we introduce some time-weighted sup-norms as follows.
\begin{eqnarray*}
\mathcal{E}_{0}(t):=\sup_{0\leq\tau\leq t}\|z(\tau)\|_{B^{3/2}_{2,1}};
\end{eqnarray*}
\begin{eqnarray*}
\mathcal{E}_{1}(t):=\sup_{0\leq\tau\leq t}(1+\tau)^{\frac{s}{2}}\|z(\tau)\|_{L^2}+\sup_{0<\ell\leq 1/2}\sup_{0\leq\tau\leq t}(1+\tau)^{\frac{s+\ell}{2}}\|z(\tau)\|_{\dot{B}^{\ell}_{2,1}}
\end{eqnarray*}
with $s\in (0,1/2]$. Let us mention that the time-weighted sup-norms are different in regard to the derivative index in $\mathcal{E}_{1}(t)$, since we take care of the sequent use of Propositions \ref{prop2.2}-\ref{prop2.3}, where the regularity index should be positive.

Precisely, we shall prove the following result.
\begin{prop}\label{prop4.1}
Let $z(t,x)$ be the global classical solution in
$\widetilde{\mathcal{C}}(B^{3/2}_{2,1}(\mathbb{R}))\cap
\widetilde{\mathcal{C}}^1(B^{1/2}_{2,1}(\mathbb{R}))$. Additional, if $z_{0}\in \dot{B}^{-s}_{2,\infty}(\mathbb{R})(1/4<s\leq 1/2)$, then
\begin{eqnarray}
\mathcal{E}_{1}(t)\lesssim E_{0}+\mathcal{E}_{1}(t)^{2}+\mathcal{E}_{0}(t)\mathcal{E}_{1}(t), \label{R-E34}
\end{eqnarray}
where $E_{0}$ is defined as Theorem \ref{thm1.1}.
\end{prop}

\begin{proof}
In \cite{XK2}, we developed time-weighted energy approaches in terms of high-frequency and low-frequency decompositions.
However, the analysis can not be copied directly in the case of one dimension. As a matter of fact, more elaborate computations are needed at the low-frequency.
Fortunately, those properties stated in Section \ref{sec:3} enable us to overcome the difficulty. Additionally, it is worth noting that
the composite functions $P$ and $Q$ both have the special form of $O(z^2)$, which is helpful to obtain desired decay estimates, see (\ref{R-E42})
and (\ref{R-E48}) below. For clarity, The proof is spitted into high-frequency and low-frequency estimates.\\

\noindent\textit{\textbf{Step 1: High-frequency estimate}}

From Propositions \ref{prop1.1}-\ref{prop1.2}, we have
\begin{eqnarray}
\|\dot{\Delta}_{j}\mathcal{G}(x,t)z_{0}\|_{L^2}\lesssim e^{-ct}\|\dot{\Delta}_{j}z_{0}\|_{L^2} \label{R-E344}
\end{eqnarray}
for $c>0$ and all $j\geq j_{0}$.

In the case of $0<\ell\leq \frac{1}{2}$, it follows from Lemma \ref{lem3.1} that
\begin{eqnarray}
&&\sum_{j\geq j_{0}}2^{j\ell}\|\dot{\Delta}_{j}z\|_{L^2}
\nonumber\\&\lesssim & e^{-ct}\|z_{0}\|_{\dot{B}^{\ell}_{2,1}}+\int^{t}_{0}e^{-c(t-\tau)}\|P_{x}+Q\|_{\dot{B}^{\ell}_{2,1}}d\tau. \label{R-E35}
\end{eqnarray} By Propositions \ref{prop2.1} and \ref{prop2.2}, we arrive at
\begin{eqnarray}
\|P_{x}\|_{\dot{B}^{\ell}_{2,1}}&=&\|D_{U}PU_{x}\|_{\dot{B}^{\ell}_{2,1}}\nonumber\\&\lesssim & \|D_{U}P\|_{\dot{B}^{\ell}_{2,1}}\|U_{x}\|_{\dot{B}^{1/2}_{2,1}}
\nonumber\\&\lesssim &\|z\|_{\dot{B}^{\ell}_{2,1}}\|z\|_{\dot{B}^{3/2}_{2,1}}\nonumber\\&\lesssim & (1+\tau)^{-\frac{s+\ell}{2}}\mathcal{E}_{0}(t)\mathcal{E}_{1}(t). \label{R-E36}
\end{eqnarray}
It follows from Proposition \ref{prop2.3} that
\begin{eqnarray}
\|Q\|_{\dot{B}^{\ell}_{2,1}}\lesssim \|z\|^2_{\dot{B}^{\ell}_{2,1}}\lesssim (1+\tau)^{-(s+\ell)}\mathcal{E}_{1}(t)^2. \label{R-E37}
\end{eqnarray}
Hence, it follows from (\ref{R-E35})-(\ref{R-E37}) that
\begin{eqnarray}
&&\sum_{j\geq j_{0}}2^{j\ell}\|\dot{\Delta}_{j}z\|_{L^2}\nonumber\\&\lesssim & e^{-ct}\|z_{0}\|_{B^{1/2}_{2,1}}+(1+t)^{-\frac{s+\ell}{2}}\mathcal{E}_{0}(t)\mathcal{E}_{1}(t)+(1+t)^{-(s+\ell)}\mathcal{E}_{1}(t)^2 \label{R-E38}
\end{eqnarray}
for $0<\ell\leq \frac{1}{2}$.

On the other hand, in the case of $\ell=0$, it follows from Lemma \ref{lem3.1} and Minkowski inequality that
\begin{eqnarray}
\|\|\dot{\Delta}_{j}z\|_{L^2}\|_{l^{2}_{j}(j\geq j_{0})}\lesssim e^{-ct}\|z_{0}\|_{L^2}+\int^{t}_{0}e^{-c(t-\tau)}\|P_{x}+Q\|_{L^2}d\tau, \label{R-E388}
\end{eqnarray}
where using the fact $P(z)=O(z^2)=Q(z)$ implies that
\begin{eqnarray}
\|P_{x}+Q\|_{L^2}\lesssim \|z(\tau)\|_{L^2}\|z(\tau)\|_{B^{3/2}_{2,1}}\lesssim (1+\tau)^{-\frac{s}{2}}\mathcal{E}_{0}(t)\mathcal{E}_{1}(t).\label{R-E389}
\end{eqnarray}
 Hence, we obtain
\begin{eqnarray}
\|\|\dot{\Delta}_{j}z\|_{L^2}\|_{l^{2}_{j}(j\geq j_{0})}\lesssim e^{-ct}\|z_{0}\|_{L^2}+(1+t)^{-\frac{s}{2}}\mathcal{E}_{0}(t)\mathcal{E}_{1}(t). \label{R-E390}
\end{eqnarray}

\noindent\textit{\textbf{Step 2: Low-frequency estimate}}

According to Lemma \ref{lem3.1}, Lemma \ref{lem3.4} and Propositions \ref{prop1.1}-\ref{prop1.2}, we obtain
\begin{eqnarray}
&&\sum_{j\leq j_{0}}2^{q\ell}\|\dot{\Delta}_{j}z\|_{L^2}\nonumber\\&\leq& \sum_{j\leq j_{0}}2^{q\ell}\|\dot{\Delta}_{j}\mathcal{G}(t)z_{0}\|_{L^2}
+\int^{t}_{0}\sum_{j\leq j_{0}}2^{q\ell}\|\dot{\Delta}_{j}\mathcal{G}(t-\tau)(A^{0})^{-1}(P_{x}+Q)(\tau)\|_{L^2}d\tau
\nonumber\\ &\lesssim & \|z_{0}\|_{\dot{B}^{-s}_{2,\infty}}(1+t)^{-\frac{s+\ell}{2}}+\int^{t}_{0}e^{-c(t-\tau)} \|Q\|_{\dot{B}^{\ell}_{2,1}}d\tau+I_{\ell}(t), \label{R-E39}
\end{eqnarray}
for $0<\ell\leq \frac{1}{2}$,
where
$$I_{\ell}(t)=\Big(\int^{t/2}_{0}+\int^{t}_{t/2}\Big)(\cdot\cdot\cdot):=\mathcal{I}_{1}(t)+\mathcal{I}_{2}(t).$$
We would like to point out the second term on the right-side of (\ref{R-E39}) is the consequence of using Lemma \ref{lem3.4}. Noticing (\ref{R-E37}), it is easy to get
\begin{eqnarray}
\int^{t}_{0}e^{-c(t-\tau)} \|Q\|_{\dot{B}^{\ell}_{2,1}}d\tau\lesssim (1+t)^{-(s+\ell)}\mathcal{E}_{1}(t)^2. \label{R-E40}
\end{eqnarray}
Additionally, for the case of $\ell=0$, we can reach a similar inequality as (\ref{R-E39})
\begin{eqnarray}
&&\|\|\dot{\Delta}_{j}z\|_{L^2}\|_{l^{2}_{j}(j\leq j_{0})}\nonumber\\&\lesssim& \|z_{0}\|_{\dot{B}^{-s}_{2,\infty}}(1+t)^{-\frac{s}{2}}+\int^{t}_{0}e^{-c(t-\tau)} \|Q\|_{L^2}d\tau+I_{0}(t), \label{R-E400}
\end{eqnarray}
where $I_{0}(t):=I_{\ell}(t)|_{\ell=0}$ and
\begin{eqnarray} \int^{t}_{0}e^{-c(t-\tau)} \|Q\|_{L^2}d\tau\lesssim (1+t)^{-\frac{s}{2}}\mathcal{E}_{0}(t)\mathcal{E}_{1}(t). \label{R-E401}  \end{eqnarray}

Next, we turn to the main estimate for $I_{\ell}(t)(0\leq\ell\leq 1/2)$. Taking $\sigma=0$ in Lemmas \ref{lem3.3}-\ref{lem3.4}, we obtain
\begin{eqnarray}
\mathcal{I}_{1}(t)\lesssim \int^{t/2}_{0}(1+t-\tau)^{-\frac{s+\ell+1}{2}}(\|P\|_{\dot{B}^{-s}_{2,\infty}}+\|Q\|_{\dot{B}^{-s}_{2,\infty}})d\tau. \label{R-E41}
\end{eqnarray}
It suffices to estimate the norm $\|P\|_{\dot{B}^{-s}_{2,\infty}}$, since both $P,Q=O(z^2)$. From the embedding property $L^{m}\hookrightarrow\dot{B}^{-s}_{2,\infty}$ in Lemma \ref{lem6.5} with $1/m=s+1/2 \ (s\in (0,1/2])$, we get
\begin{eqnarray}
\|P\|_{\dot{B}^{-s}_{2,\infty}}\lesssim \|P\|_{L^{m}} \lesssim \|z\|_{L^{1/s}}\|z\|_{L^2}. \label{R-E42}
\end{eqnarray}
Using interpolation inequalities in Lemma \ref{lem6.4}, we are led to
\begin{eqnarray}
\|z\|_{L^{1/s}}\|z\|_{L^2}&\lesssim& \|\Lambda^{\varepsilon}z\|^{\theta}_{L^2}\|\Lambda^{\alpha}z\|^{1-\theta}_{L^2}\|z\|_{L^2}\nonumber\\ &\lesssim &
(\|\Lambda^{\varepsilon}z\|_{L^2}+\|\Lambda^{\alpha}z\|_{L^2})\|z\|_{L^2}, \label{R-E43}
\end{eqnarray}
where the parameter couple $(\varepsilon,\alpha)$ is subjected to the constraint
$1/2-s\leq 1-2s\leq\varepsilon<\alpha\leq1/2$. Clearly, $s\in(1/4,1/2]$ is additionally needed. In this case, let us note that
$\theta=\frac{\alpha+s-1/2}{\alpha-\varepsilon}$ in (\ref{R-E43}). Furthermore, it follows from the definition of time-weighted energy functional that
\begin{eqnarray}
\|P\|_{\dot{B}^{-s}_{2,\infty}}&\lesssim &\Big[(1+\tau)^{-\frac{s+\varepsilon}{2}}+(1+\tau)^{-\frac{s+\alpha}{2}}\Big](1+\tau)^{-\frac{s}{2}}\mathcal{E}_{1}(t)^2
\nonumber\\ &\lesssim & (1+\tau)^{-s-\frac{\varepsilon}{2}}\mathcal{E}_{1}(t)^2. \label{R-E44}
\end{eqnarray}
Similarly, it also holds that
\begin{eqnarray}
\|Q\|_{\dot{B}^{-s}_{2,\infty}}\lesssim (1+\tau)^{-s-\frac{\varepsilon}{2}}\mathcal{E}_{1}(t)^2. \label{R-E45}
\end{eqnarray}
Then we conclude that
\begin{eqnarray}
\mathcal{I}_{1}&\lesssim&  \mathcal{E}_{1}(t)^2 \int^{t/2}_{0}(1+t-\tau)^{-\frac{s+\ell+1}{2}} (1+\tau)^{-s-\frac{\varepsilon}{2}}d\tau
\nonumber\\ &\lesssim &
\mathcal{E}_{1}(t)^2(1+t)^{-\frac{s+\ell+1}{2}}\int^{t/2}_{0}(1+\tau)^{-s-\frac{\varepsilon}{2}}d\tau
\nonumber\\ &\lesssim &\mathcal{E}_{1}(t)^2(1+t)^{-\frac{s+\ell}{2}}(1+t)^{\frac{1}{2}-s-\frac{\varepsilon}{2}}
\nonumber\\ &\lesssim & \mathcal{E}_{1}(t)^2 (1+t)^{-\frac{s+\ell}{2}}, \label{R-E46}
\end{eqnarray}
where we used the constraints $s+\frac{\varepsilon}{2}\in [1/2,3/4)$ and $1-2s\leq\varepsilon$.

By taking $\sigma=s+1/2$ in Lemmas \ref{lem3.3}-\ref{lem3.4}, we arrive at
\begin{eqnarray}
\mathcal{I}_{2}&\lesssim & \int^{t}_{t/2}(1+t-\tau)^{-\frac{\ell+1/2}{2}}(\|\Lambda^{s+\frac{1}{2}}P\|_{\dot{B}^{-s}_{2,\infty}}+\|\Lambda^{s+\frac{1}{2}}Q\|_{\dot{B}^{-s}_{2,\infty}})d\tau,
\nonumber\\&=:& \mathcal{I}_{21}+\mathcal{I}_{22}, \label{R-E47}
\end{eqnarray}
where using Proposition \ref{prop2.3} gives
\begin{eqnarray}
\mathcal{I}_{21}&\lesssim&\int^{t}_{t/2}(1+t-\tau)^{-\frac{\ell+1/2}{2}}\|P\|_{\dot{B}^{1/2}_{2,\infty}}d\tau
\nonumber\\&\lesssim& \int^{t}_{t/2}(1+t-\tau)^{-\frac{\ell+1/2}{2}}\|z\|^2_{\dot{B}^{1/2}_{2,\infty}}d\tau
\nonumber\\ &\lesssim & \int^{t}_{t/2}(1+t-\tau)^{-\frac{\ell+1/2}{2}}\|z\|^2_{\dot{B}^{1/2}_{2,1}}d\tau
\nonumber\\ &\lesssim & \mathcal{E}_{1}(t)^2 \int^{t}_{t/2}(1+t-\tau)^{-\frac{\ell+1/2}{2}}(1+\tau)^{-1}d\tau
\nonumber\\ &\lesssim & \mathcal{E}_{1}(t)^2 (1+t)^{-1}\int^{t}_{t/2}(1+t-\tau)^{-\frac{\ell+1/2}{2}}d\tau. \label{R-E48}
\end{eqnarray}
Due to the fact $\ell\in [0,1/2]$, we know that $\frac{\ell+1/2}{2}\in [1/4,1/2]$.
Furthermore,
\begin{eqnarray}
\mathcal{I}_{21}&\lesssim& \mathcal{E}_{1}(t)^2 (1+t)^{-\frac{\ell+1/2}{2}} \nonumber\\ &\lesssim &
\mathcal{E}_{1}(t)^2 (1+t)^{-\frac{s+\ell}{2}}. \label{R-E49}
\end{eqnarray}
Similarly,
\begin{eqnarray}
\mathcal{I}_{22}\lesssim
\mathcal{E}_{1}(t)^2 (1+t)^{-\frac{s+\ell}{2}}. \label{R-E50}
\end{eqnarray}
Hence, together with (\ref{R-E49})-(\ref{R-E50}),  we can deduce that
\begin{eqnarray}
\mathcal{I}_{2}\lesssim
\mathcal{E}_{1}(t)^2 (1+t)^{-\frac{s+\ell}{2}}. \label{R-E51}
\end{eqnarray}
Therefore, combining (\ref{R-E46}) and (\ref{R-E51}), the low-frequency estimates read as
\begin{eqnarray}
\sum_{j\leq j_{0}}2^{q\ell}\|\dot{\Delta}_{j}z\|_{L^2}\lesssim\|z_{0}\|_{\dot{B}^{-s}_{2,\infty}}(1+t)^{-\frac{s+\ell}{2}}+\mathcal{E}_{1}(t)^2 (1+t)^{-\frac{s+\ell}{2}} \label{R-E52}
\end{eqnarray}
for $0<\ell\leq1/2$ and
\begin{eqnarray}
&&\|\|\dot{\Delta}_{j}z\|_{L^2}\|_{l^{2}_{j}(j\leq j_{0})}\nonumber\\&\lesssim& \|z_{0}\|_{\dot{B}^{-s}_{2,\infty}}(1+t)^{-\frac{s}{2}}+\mathcal{E}_{0}(t)\mathcal{E}_{1}(t)(1+t)^{-\frac{s}{2}}+\mathcal{E}_{1}(t)^2 (1+t)^{-\frac{s}{2}}. \label{R-E522}
\end{eqnarray}

Finally, we combine high-frequency estimates (\ref{R-E38}), (\ref{R-E390}) and low-frequency estimates (\ref{R-E52})-(\ref{R-E522}) to get
\begin{eqnarray}
\mathcal{E}_{1}(t)\lesssim E_{0}+\mathcal{E}_{0}(t)\mathcal{E}_{1}(t)+\mathcal{E}_{1}(t)^{2}, \label{R-E53}
\end{eqnarray}
which is just (\ref{R-E34}).

Hence, the proof of Proposition \ref{prop4.1} is finished.
\end{proof}

By virtue of the crucial time-weighted estimate (\ref{R-E34}), the proof of Theorem \ref{thm1.1} can be shown as follows.\\

\noindent \textbf{\textit{The proof of Theorem \ref{thm1.1}}.}
According to the global-in-time result in \cite{XK1}, we see that
$\mathcal{E}_{0}(t)\lesssim \|U_{0}-\bar{U}\|_{B^{3/2}_{2,1}(\mathbb{R})}\leq E_{0}$. Thus, if $E_{0}$ is sufficient small, it follows from (\ref{R-E34}) that
\begin{eqnarray}
\mathcal{E}_{1}(t)\lesssim E_{0}+\mathcal{E}_{1}(t)^{2}, \label{R-E54}
\end{eqnarray}
which leads to $\mathcal{E}(t)\lesssim E_{0}$ by the standard argument, provided that $E_{0}$ is sufficient small. Consequently, we obtain desired decay
estimates in Theorem \ref{thm1.1}. $\square$

\section{Applications} \setcounter{equation}{0}\label{sec:5}
In this section, we focus on a number of  hyperbolic systems subjected to the same dissipative structure as (\ref{R-E1}) satisfying the Kawashima-Shizuta condition, for instance,  damped compressible Euler equations, Thermoelasticity with second sound and Timoshenko systems with equal speeds. Please allow us to abuse notations a little on the statement of main results.

\subsection{Damped compressible Euler equations}\setcounter{equation}{0}

Consider the following damped Euler equations for a perfect flow
\begin{equation}
\left\{
\begin{array}{l}
\partial_{t}\rho + \nabla\cdot(\rho\textbf{u}) = 0 , \\
\partial_{t}(\rho\textbf{u}) +\nabla\cdot(\rho\textbf{u}\otimes\textbf{u}) +
\nabla p(\rho) =-\rho\textbf{u}.
\end{array} \right.\label{R-E55}
\end{equation}
Here $\rho = \rho(t, x)$ is the fluid density function of
$(t,x)\in[0,+\infty)\times\mathbb{R}^{n}$;
$\textbf{u}=\textbf{u}(t, x)=(u_{1},u_{2},\cdot\cdot\cdot,u_{n})^{\top}$ denotes the
fluid velocity. The pressure
$p(\rho)$  satisfies the classical assumption
$$p'(\rho)>0,\ \ \ \forall\rho>0.$$
The notation $\nabla,\otimes$ are the gradient operator (in $x$) and
the symbol for the tensor products of two vectors, respectively.

System (\ref{R-E55}) is complemented by the initial conditions
\begin{equation}
(\rho,\textbf{u})(0,x)=(\rho_{0},\textbf{u}_{0}).\label{R-E555}
\end{equation}
Let $(\bar{\rho},0)$ be the reference constant equilibrium. Set $\upsilon=\rho\textbf{u}/\bar{\rho}$. It is not difficult to rewrite (\ref{R-E55}) as
\begin{equation}
\left\{
\begin{array}{l}
\partial_{t}\rho + \bar{\rho}\mathrm{div}\upsilon = 0 , \\
\partial_{t}\upsilon+ \bar{a}\nabla\rho+\upsilon=\mathrm{div}q_{2}/\bar{\rho},
\end{array} \right.\label{R-E56}
\end{equation}
where $\bar{a}=p'(\bar{\rho})/\bar{\rho}$ and
$$q_{2}=-\bar{\rho}^2\upsilon\otimes \upsilon/\rho-[p(\rho)-p(\bar{\rho})-p'(\bar{\rho})(\rho-\bar{\rho})]I_{n}. $$
We put $w:=(\rho-\bar{\rho}, \upsilon)^{\top}$. The initial data read correspondingly as
\begin{equation}
w|_{t=0}=(\rho_{0}-\bar{\rho}, \upsilon_{0})^{\top}(x) \label{R-E57}
\end{equation}
with $\upsilon_{0}=\rho_{0}\textbf{u}_{0}/\rho_{\infty}$.
System (\ref{R-E56}) is also rewritten in the vector form
\begin{equation}
A^{0}w_{t}+\sum_{j=1}^{n}A^{j}w_{x_{j}}+Lw=\mathrm{div}q, \label{R-E58}
\end{equation}
where
$$A^{0}=\left(
          \begin{array}{cc}
            \bar{a} & 0 \\
            0 & \bar{\rho}I_{n} \\
          \end{array}
        \right),\ \quad A^{j}=\left(
                                \begin{array}{cc}
                                  0 & \bar{a}\bar{\rho}e_{j}^\top \\
                                  \bar{a}\bar{\rho}e_{j} & 0 \\
                                \end{array}
                              \right),\ \quad L=\left(
                                                  \begin{array}{cc}
                                                    0 & 0 \\
                                                    0 & \bar{\rho}I_{n} \\
                                                  \end{array}
                                                \right)
$$
and $q=(0,q_{2}/\bar{\rho})^\top$. Note that $q=O(|w|^2)$.
Here $I_{n}$ denotes the unit matrix of order $n$ and $e_{j}$ is
$n$-dimensional vector where the $j$th component is one, others are
zero.

System (\ref{R-E55}) describes that the compressible gas flow passes a porous medium and the medium induces a
friction force, proportional to the linear momentum in the opposite direction. So far there are many efforts available
for the damped Euler equations by various authors. In one space dimension in Lagrangian
coordinates, Nishida \cite{N2} obtained the global classical
solutions with small data, and the solutions following Darcy's law
asymptotically as time tends to infinity was shown by Hsiao and Liu
\cite{HL}. Nishihara et al. \cite{N,NWY} proved its optimal convergence rate. Later, Nishihara and Yang \cite{NY} investigated the boundary effect on the asymptotic behavior of solutions. Huang et al. \cite{HP2,HMP} established the large-time behavior of $L^\infty$ entropy weak solutions with vacuum.
In higher dimensions, Wang and Yang \cite{WY} showed the global existence and pointwise estimates
of the solutions by a detailed analysis of the Green function.  Sideris et al. \cite{STW}
proved that the damping term could prevent the development singularities if the initial data is small in an appropriate
norm, furthermore, it was shown that the classical solution decays in the $L^2$-norm
to the constant background state at the rate of $(1+t)^{-3/4}$. Tan and Wu \cite{TW2} performed the elaborate spectral analysis
to improve those decay rates in \cite{STW} such that the density converges to its equilibrium state at the rates $(1+t)^{-\frac{3}{4}-\frac{s}{2}}$ in the $L^2$-norm, and the momentum of the system decays at the rates $(1+t)^{-\frac{5}{4}-\frac{s}{2}}$ in the $L^2$-norm, as the initial data
$(\rho_{0},\textbf{u}_{0})\in H^{l}(\mathbb{R}^{3})\cap \dot{B}^{-s}_{1,\infty}(\mathbb{R}^{3}) (l\geq4,\ s\in [0,1])$. In \cite{XK2}, we proved
the optimal decay rates on the framework of spatially critical Besov spaces $B^{s_{c}}_{2,1}\cap \dot{B}^{-s}_{2,\infty}\ (s_{c}=1+n/2,\ s\in (0,n/2], n\geq3)$. In addition, Tan and Wang \cite{TW1} adopted a different method to obtain the optimal decay rates in $\mathbb{R}^{3}$, whose idea was a family of scaled energy estimates with minimum derivative counts and interpolations among them without linear decay analysis. It should be noted that
above decay results hold true for higher dimensions $(n\geq3)$ due to the usage of Gagliardo-Nirenberg-Sobolev inequalities, for example, $\|f\|_{L^{2^{*}}}\lesssim \|\nabla f\|_{L^2}$ with $2^{*}=2n/(n-2)$.

In this paper, as an application of general results for hyperbolic systems of balance laws, we give the decay rates for (\ref{R-E55}) in lower dimensions, which can be viewed as a supplement of that in \cite{XK2}. Indeed, compared with the general form (\ref{R-E3}), the non-degenerate nonlinear variable in (\ref{R-E56}) is absent, which enables us to solve the whole decay problem in not only dimension 1 but dimension 2 as well,  following from the proof of Theorem \ref{thm1.1}. Main results are stated as follows $(s_{c}=1+n/2)$.

\begin{thm}\label{thm4.1} Let $n=1,2$ and $w:=(\rho-\bar{\rho}, \upsilon)^{\top}$ be the global classical solution in $\widetilde{\mathcal{C}}(B^{s_{c}}_{2,1}(\mathbb{R}^{n}))\cap\widetilde{\mathcal{C}}^1(B^{s_{c}-1}_{2,1}(\mathbb{R}^{n}))$ constructed by \cite{XK1}. Suppose that $w_{0}\in B^{s_{c}}_{2,1}(\mathbb{R}^{n})\cap \dot{B}^{-s}_{2,\infty}(\mathbb{R}^{n})\ (1/2-n/4<s\leq1/2)$ and the norm
$\mathcal{M}_{0}:=\|w_{0}\|_{B^{s_{c}}_{2,1}(\mathbb{R}^{n})\cap \dot{B}^{-s}_{2,\infty}(\mathbb{R}^{n})}$ is sufficiently small. Then it holds that
\begin{eqnarray}
\|\Lambda^{\ell}w(t)\|_{X_{1}}\lesssim \mathcal{M}_{0}(1+t)^{-\frac{s+\ell}{2}} \label{R-E59}
\end{eqnarray}
for $0\leq\ell\leq1/2$, where $X_{1}:=B^{s_{c}-1-\ell}_{2,1}(\mathbb{R}^{n})$ if $0\leq\ell<s_{c}-1$ and $X_{1}:=\dot{B}^{0}_{2,1}(\mathbb{R}^{n})$ if $\ell=s_{c}-1$.
In particular, one has
\begin{eqnarray}
\|\Lambda^{\ell}w(t)\|_{L^2(\mathbb{R}^{n})}\lesssim \mathcal{M}_{0}(1+t)^{-\frac{s+\ell}{2}} \label{R-E60}
\end{eqnarray}
for $0\leq\ell\leq s_{c}-1$.
\end{thm}
\begin{proof}
As shown by \cite{XK1, XK2}, the linearized system of (\ref{R-E58}) satisfies the Kawashima-Shizuta condition and admits the dissipative structure
\begin{eqnarray}
\frac{d}{dt}E[\hat{w}]+c\eta_{1}(\xi)|\hat{w}|^2\leq0, \label{R-E61}
\end{eqnarray}
for $c>0$, where $E[\hat{w}]\approx|\hat{w}|^2$ and $\eta_1(\xi)=\frac{|\xi|^2}{1+|\xi|^2}$. Hence, we can have the similar decay properties as Propositions \ref{prop1.1}-\ref{prop1.2}.
  Denote
\begin{eqnarray}
\widehat{\mathcal{G}f}(t,\xi)=e^{t\Phi(i\xi)}\hat{f}(\xi), \label{R-E62}
\end{eqnarray}
where $$\Phi(i\xi)=-(A^{0})^{-1}[A(i\xi)+L]$$ with
$A(i\xi)=i\sum_{j=1}^{n}A^{j}\xi_{j}$. Then $\mathcal{G}(t,x)w_{0}$ is the solution of linearized system
\begin{equation}
\left\{
\begin{array}{l}
A^{0}w_{t}+\sum_{j=1}^{n}A^{j}w_{x_{j}}+Lw=0,\\
w|_{t=0}=w_{0}.
\end{array}\right. \label{R-E63}
\end{equation}
Define the time-weighted sup-norms as follows
\begin{eqnarray*}
\mathcal{\tilde{E}}_{0}(t):=\sup_{0\leq\tau\leq t}\|w(\tau)\|_{B^{s_{c}}_{2,1}};
\end{eqnarray*}
\begin{eqnarray*}
\mathcal{\tilde{E}}_{1}(t):&=&\sup_{0\leq\ell<s_{c}-1}\sup_{0\leq\tau\leq t}(1+\tau)^{\frac{s+\ell}{2}}\|\Lambda^{\ell}w(\tau)\|_{B^{s_{c}-1-\ell}_{2,1}}\nonumber\\&&+\sup_{0\leq\tau\leq t}(1+\tau)^{\frac{s+s_{c}-1}{2}}\|\Lambda^{s_{c}-1}w(\tau)\|_{\dot{B}^{0}_{2,1}}.
\end{eqnarray*}
Note that the time-weighted sup-norms are also different in regard to the derivative index. Actually, by (1) in Lemma \ref{lem2.2}, the sup-norm $\mathcal{\tilde{E}}_{1}(t)$ is stronger than $\mathcal{E}_{1}(t)$ in the proof of Theorem \ref{thm1.1}, which leads to develop the frequency-localization Duhamel principle (Lemma \ref{lem3.2}) in the inhomogeneous case.

Next, we give the outline of decay estimates only, since it is similar to the proof of Theorem 1.1.
Due to the fact $\Delta_{j}f\equiv\dot{\Delta}_{j}f(j\geq0$), it suffices to give the high-frequency estimate for the inhomogeneous case
\begin{eqnarray}
\sum_{j\geq 0}2^{j(1/2-\ell)}\|\Delta_{j}\Lambda^{\ell}w\|_{L^2}\lesssim e^{-ct}\|w_{0}\|_{B^{1/2}_{2,1}}+(1+t)^{-\frac{s+\ell}{2}}\mathcal{E}_{0}(t)\mathcal{E}_{1}(t) \label{R-E64}
\end{eqnarray}
for $0\leq\ell\leq s_{c}-1$. Indeed,  this follows from Lemma \ref{lem3.1} and Proposition \ref{prop1.1} directly.

For the low-frequency estimates, by Lemma \ref{lem3.1}, we arrive at
\begin{eqnarray}
&&\|\Delta_{-1}\Lambda^{\ell}w(t,x)\|_{L^2}\nonumber\\&\leq&\|\Delta_{-1}\Lambda^{\ell}[\mathcal{G}(t)w_{0}]\|_{L^2}
+\int^{t}_{0}\|\Delta_{-1}\Lambda^{\ell}[\mathcal{G}(t-\tau)\mathrm{div}q(\tau)]\|_{L^2}d\tau
\nonumber\\&\lesssim&\|w_{0}\|_{\dot{B}^{-s}_{2,\infty}}(1+t)^{-\frac{s+\ell}{2}}
+\tilde{I}_{1}+\tilde{I}_{2}, \label{R-E65}
\end{eqnarray}
for $0\leq\ell<s_{c}-1$, where
\begin{eqnarray*}
\tilde{I}_{1}=\int_{0}^{t/2}\|\Delta_{-1}\Lambda^{\ell}[\mathcal{G}(t-\tau)A_{0}^{-1}\mathrm{div}q(\tau)]\|_{L^2} d\tau,
\end{eqnarray*}
and
\begin{eqnarray*}
\tilde{I}_{2}=\int_{t/2}^{t}\|\Delta_{-1}\Lambda^{\ell}[\mathcal{G}(t-\tau)A_{0}^{-1}\mathrm{div}q(\tau)]\|_{L^2} d\tau.
\end{eqnarray*}
Thanks to Lemma \ref{lem3.2}, just proceeding with the same procedure leading to (\ref{R-E46}) and (\ref{R-E51}), we conclude that
\begin{eqnarray}
(\tilde{I}_{1},\tilde{I}_{2})\lesssim \mathcal{\tilde{E}}^2_{1}(t)(1+t)^{-\frac{s+\ell}{2}}. \label{R-E66}
\end{eqnarray}
Here, let us explain a little for the inequality (\ref{R-E66}). When estimating $\tilde{I}_{1}$, $\sigma=0$ is chosen in Lemma \ref{lem3.2}.
The parameter couple $(\varepsilon,\alpha)$ in the usage of interpolation inequalities in Lemma \ref{lem6.4} is subjected to the constraint
$n(1/2-s)\leq 1-2s\leq\varepsilon<\alpha\leq n/2\, (n=1,2)$. Therefore, $s\in(1/2-n/4,1/2]$ is needed to be restricted. In this case, the power index $\theta$ satisfies $\theta=\frac{\alpha-n(1/2-s)}{\alpha-\varepsilon}$.
On the other hand,  $\sigma=s+1/2$ is taken in Lemma \ref{lem3.2} when estimating $\tilde{I}_{2}$.

 In the case of $\ell=s_{c}-1$, similar to the procedure leading to (\ref{R-E52}), we can deduce that
\begin{eqnarray}
&&\sum_{j<0}\|\dot{\Delta}_{j}\Lambda^{s_{c}-1}w\|_{L^2}
\nonumber\\&\leq&\sum_{j<0}
\|\dot{\Delta}_{j}\Lambda^{s_{c}-1}[\mathcal{G}(t)w_{0}]\|_{L^2}
+\int^{t}_{0}\sum_{j<0}\|\dot{\Delta}_{j}\Lambda^{s_{c}-1}[\mathcal{G}(t-\tau)\mathrm{div}q(\tau)]\|_{L^2}d\tau
\nonumber\\&\lesssim &  \|w_{0}\|_{\dot{B}^{-1/2}_{2,\infty}} (1+t)^{-\frac{s+s_{c}-1}{2}}+\mathcal{\tilde{E}}_{1}(t)^2(1+t)^{-\frac{s+s_{c}-1}{2}}. \label{R-E67}
\end{eqnarray}
With these preparations (\ref{R-E64})-(\ref{R-E65}), and (\ref{R-E66})-(\ref{R-E67}) in hand, it follows from the definition of $\mathcal{\tilde{E}}_{1}(t)$ that
\begin{eqnarray}
\mathcal{\tilde{E}}_{1}(t)\lesssim \mathcal{M}_{0}+\mathcal{\tilde{E}}_{1}(t)^{2}+\mathcal{\tilde{E}}_{0}(t)\mathcal{\tilde{E}}_{1}(t). \label{R-E68}
\end{eqnarray}
which implies that (\ref{R-E59}) readily. (\ref{R-E60}) is followed by $\dot{B}^{0}_{2,1}\hookrightarrow L^2$ and (1) in Lemma \ref{lem2.2}.
\end{proof}

As a consequence, the usual optimal decay estimates of $L^1(\mathbb{R}^{n})$-$L^2(\mathbb{R}^{n})$ is available.
\begin{cor} \label{cor6.1}
Let $n=1,2$ and $w:=(\rho-\bar{\rho}, \upsilon)^{\top}$ be the global classical solution in $\widetilde{\mathcal{C}}(B^{s_{c}}_{2,1}(\mathbb{R}^{n}))\cap\widetilde{\mathcal{C}}^1(B^{s_{c}-1}_{2,1}(\mathbb{R}^{n}))$.
Suppose that $w_{0}\in B^{s_{c}}_{2,1}(\mathbb{R}^{n})\cap L^1(\mathbb{R}^{n})$ and the norm
$\mathcal{\widetilde{M}}_{0}:=\|w_{0}\|_{B^{s_{c}}_{2,1}(\mathbb{R}^{n})\cap L^1(\mathbb{R}^{n})}$ is sufficiently small. Then it holds that
\begin{eqnarray}
\|\Lambda^{\ell}w\|_{L^2(\mathbb{R}^{n})}\lesssim \mathcal{\widetilde{M}}_{0} (1+t)^{-\frac{n}{4}-\frac{\ell}{2}}, \ \ \  0\leq \ell\leq s_{c}-1. \label{R-E69}
\end{eqnarray}
\end{cor}

Finally, we would like to mention that another interesting line of research of (\ref{R-E55}), say, to justify the diffusive relation limit.
For entropy weak solutions, the paper of Marcati and Milani \cite{MM} was concerned with the pourous media flow
as the limit of the 1-D Euler equation, later generalized by Marcati and Rubino
\cite{MR} to the multi-D case. Junca and Rascle \cite{JR} proved the convergence for arbitrarily large $BV(\mathbb{R})$ solution away from vacuum.
For smooth solutions, the diffusive limit was justified by Coulombel et al. \cite{CG,LC} in Sobolev spaces with higher regularity, and by the first author et. al.  \cite{XK3,XW} in Besov spaces with critical regularity. Besides, inspired by the formal Maxwell iteration, the first author \cite{X} obtained an definite order for this diffusive convergence.

\subsection{Thermoelasticity with second sound}
Consider the quasi-linear equations in thermoelasticity:
\begin{equation}
\left\{
\begin{array}{l}
u_{tt}-\psi(u_{x})_{x}+\beta\theta_{x}=0,\\
\theta_{t}+\eta q_{x}+\delta u_{tx}=0,\\
\tau q_{t}+q+\kappa \theta_{x}=0,
\end{array}\right. \label{R-E70}
\end{equation}
where $\psi(r)$ is assumed to be a smooth function of $r$ such that $\psi'(r)>0$. This system describes the propagation of nonlinear elastic waves in the presence of thermoelastic effects. Here $t\geq 0$ is the time variable and $x\in \mathbb{R}$ is the spacial
variable. $u(t,x)$ stands for the deformation, $\theta(t,x)$ is the temperature field and $q(t,x)$ the heat flux. $\beta, \eta, \delta$ and $\tau$ are positive physical constants, for simplicity, which are normalized to be one in this paper.

System (\ref{R-E70}) is supplemented with the initial data
\begin{equation}
u|_{t=0}=u_{0},\quad  u_{t}|_{t=0}=u_{1},\quad \theta|_{t=0}=\theta_{0},\quad q|_{t=0}=q_{0}. \label{R-E71}
\end{equation}
As in \cite{KRH}, we introduce
the quantities
\begin{equation}\label{R-E72}
v=u_t, \quad
r=u_x. \quad
\end{equation}
Then, (\ref{R-E70}) can be rewritten as the form of first order
\begin{equation}
\left\{
\begin{array}{l}
v_{t}-\psi(r)_{x}+\theta_{x}=0,\\
r_{t}-v_{x}=0,\\
\theta_{t}+q_{x}+v_{x}=0,\\
q_{t}+q+\theta_{x}=0.
\end{array}\right. \label{R-E73}
\end{equation}
The initial data are given by
\begin{equation}\label{R-E74}
(v, r, \theta, q)|_{t=0}
=(v_{0}, r_{0}, \theta_{0}, q_{0})(x)
\end{equation}
with $v_0=u_{1}$ and $r_0=u_{x0}$.

Set $W=(v, r, \theta, q)^{\top}$.
It is convenient to transform
(\ref{R-E70})-(\ref{R-E71}) into a Cauchy problem for the hyperbolic form
\begin{equation}\label{R-E75}
\left\{\begin{array}{l}
       A_{0}W_t + A_{1}W_x + LW =p_{x},\\[2mm]
       W(x, 0) = W_0(x),
\end{array}\right.
\end{equation}
where
\begin{eqnarray*}
A_{0}=\left(
\begin{array}{cccc}
1 & 0 & 0 & 0 \\
0 & a^{2} & 0 & 0 \\
0 & 0 & 1 & 0 \\
0 & 0 & 0 & 1
\end{array}
\right),\ \
A_{1}=-\left(
\begin{array}{cccc}
0 & a^{2} & 0 & 0 \\
a^{2} & 0 & 0 & 0 \\
0 & 0 & 0 & 1 \\
0 & 0 & 1 & 0
\end{array}
\right),\ \
L=\left(
\begin{array}{cccc}
0 & 0 & 0 & 0 \\
0 & 0 & 0 & 0 \\
0 & 0 & 0 & 0 \\
0 & 0 & 0 & 1
\end{array}
\right)
\end{eqnarray*}
and
\begin{eqnarray*}
p=(p_{1},0,0,0)^{\top}
\end{eqnarray*}
with $p_{1}=\psi(r)-\psi(0)-a^2r=O(r^2)$ and $\psi'(0)=a^2$.

If the relaxation time $\tau$ vanishes in (\ref{R-E70}), at the formal level, we see that (\ref{R-E70}) reduces to the classical Fourier's law. If
$\tau>0$, the heat conduction is modelling by Cattaneo's law. Wang et al. \cite{YW} obtained the $L^p$-$L^q$ decay estimates for the linearized thermoelastic system. Subsequently, Racke and Wang \cite{RW} considered the nonlinear thermoelastic system with more general nonlinearities and derived the decay rates. Tarabek \cite{T} investigated a class of quasilinear thermoelastic system in both bounded and unbounded domains and established the global existence result for small initial data, however, the decay rates are absent. Recently, Kasimov, Racke and Said-Houari \cite{KRH} improved those decay properties in \cite{YW} such that linearized solutions decay faster with a rate of $t^{-\gamma/2}$ with the additional assumption $\int_{\mathbb{R}}U_{0}dx=0$, by introducing the integral space $L^{1,\gamma}(\mathbb{R})$. Furthermore, they deduced the optimal decay rates for (\ref{R-E70}) by employing time-weighted energy approaches in \cite{Ma}, if the initial data $W_{0}\in H^{s}(\mathbb{R})\cap L^{1}(\mathbb{R})(s\geq 3)$. Additionally,
the formation of singularities in thermoelasticity with second sound was recently investigated in \cite{HR}.

Observe that the symmetric system in (\ref{R-E75}) satisfies the Kawashima-Shizuta condition in \cite{SK} and the degenerate matrix $L$ is also symmetric. Therefore,
(\ref{R-E75}) is included in the general framework investigated by \cite{XK1}. As the direct application, we obtain the following global-in-time result in spatially critical Besov spaces.

\begin{thm}\label{thm5.2}
Suppose $W_{0}\in B^{3/2}_{2,1}(\mathbb{R})$. There exists a constant $\delta_{0}>0$ such that
if
\begin{eqnarray*}
\|W_{0}\|_{B^{3/2}_{2,1}(\mathbb{R})}\leq \delta_{0},
\end{eqnarray*}
then (\ref{R-E75}) admits a unique
global solution $W(t,x)\in \mathcal{C}^{1}(\mathbb{R}^{+}\times
\mathbb{R})$ satisfying
\begin{eqnarray*}
W(t,x)\in
\widetilde{\mathcal{C}}(B^{3/2}_{2,1}(\mathbb{R}))\cap
\widetilde{\mathcal{C}}^1(B^{1/2}_{2,1}(\mathbb{R})).
\end{eqnarray*}
Moreover, the energy inequality holds that
\begin{eqnarray}
&&\|W\|_{\widetilde{L}^\infty(B^{3/2}_{2,1}(\mathbb{R}))}
+\mu_{0}\Big(\|q\|_{\widetilde{L}^2(B^{3/2}_{2,1}(\mathbb{R}))}
+\|\nabla
W\|_{\widetilde{L}^2(B^{1/2}_{2,1}(\mathbb{R}))}\Big)
\nonumber\\&\leq&
C_{0}\|W_{0}\|_{B^{3/2}_{2,1}(\mathbb{R})}.
\label{R-E76}
\end{eqnarray}
for positive constants $C_{0}$ and $\mu_{0}$.
\end{thm}

\begin{rem}
Theorem \ref{thm1.1} exhibits the optimal critical regularity of global existence of classical solutions to (\ref{R-E75}), which improves the result in \cite{KRH}, where $s\geq2$. From (\ref{R-E76}), we see that the variable $q$ in (\ref{R-E75}) is non-degenerately dissipative, whereas other variables are degenerately dissipative.
\end{rem}

By employing the energy method in Fourier spaces in \cite{KRH}, it is known that the linearized system
\begin{equation}\label{R-E755}
\left\{\begin{array}{l}
       A_{0}W_t + A_{1}W_x + LW =0,\\[2mm]
       W(x, 0) = W_0(x),
\end{array}\right.
\end{equation}
satisfies the differential inequality
\begin{eqnarray}
\frac{d}{dt}E[\hat{W}]+c\eta_{1}(\xi)|\hat{W}|^2\leq0, \label{R-E77}
\end{eqnarray}
for $c>0$, where $E[\hat{W}]\approx|\hat{W}|^2$ and $\eta_1(\xi)=\frac{\xi^2}{1+\xi^2}$. Hence,  (\ref{R-E755}) enjoys the similar decay properties as in Propositions \ref{prop1.1}-\ref{prop1.2}. Noticing that the special nonlinear structure in (\ref{R-E75}), we define the same time-weighted energy functionals as the compressible Euler equations (\ref{R-E58})
\begin{eqnarray*}
\mathcal{\tilde{E}}_{0}(t):=\sup_{0\leq\tau\leq t}\|W(\tau)\|_{B^{3/2}_{2,1}};
\end{eqnarray*}
\begin{eqnarray*}
\mathcal{\tilde{E}}_{1}(t):&=&\sup_{0\leq\ell<1/2}\sup_{0\leq\tau\leq t}(1+\tau)^{\frac{s+\ell}{2}}\|\Lambda^{\ell}W(\tau)\|_{B^{1/2-\ell}_{2,1}}\nonumber\\&&+\sup_{0\leq\tau\leq t}(1+\tau)^{\frac{s+1/2}{2}}\|\Lambda^{\frac{1}{2}}W(\tau)\|_{\dot{B}^{0}_{2,1}}.
\end{eqnarray*}

Then proceeding the same localized time-weighted energy approaches, we obtain the following decay rates of classical solutions.
\begin{thm}\label{thm4.1}
Let $W=(v,r,\theta,q)^{\top}$ be the global classical solution in the sense of Theorem \ref{thm5.2}. Suppose that $U_{0}\in B^{3/2}_{2,1}(\mathbb{R})\cap \dot{B}^{-s}_{2,\infty}(\mathbb{R}) (1/4<s\leq1/2)$ and the norm
$\mathcal{M}_{0}:=\|W_{0}\|_{B^{3/2}_{2,1}(\mathbb{R})\cap \dot{B}^{-s}_{2,\infty}(\mathbb{R})}$ is sufficiently small. Then it holds that
\begin{eqnarray}
\|\Lambda^{\ell}W(t)\|_{X_{2}}\lesssim \mathcal{M}_{0}(1+t)^{-\frac{s+\ell}{2}} \label{R-E78}
\end{eqnarray}
for $0\leq\ell\leq1/2$, where $X_{2}:=B^{1/2-\ell}_{2,1}(\mathbb{R})$ if $0\leq\ell<1/2$ and $X_{2}:=\dot{B}^{0}_{2,1}(\mathbb{R})$ if $\ell=1/2$.
In particular, one has
\begin{eqnarray}
\|\Lambda^{\ell}W(t)\|_{L^2(\mathbb{R})}\lesssim \mathcal{M}_{0}(1+t)^{-\frac{s+\ell}{2}} \label{R-E79}
\end{eqnarray}
for $0\leq\ell\leq1/2$.
\end{thm}

As a consequence, we have
\begin{cor}\label{cor5.2}
Let $W=(v,r,\theta,q)^{\top}$ be the global classical solution in the sense of Theorem \ref{thm5.2}. Suppose that $W_{0}\in B^{3/2}_{2,1}(\mathbb{R})\cap L^1(\mathbb{R})$ and the norm
$\mathcal{\widetilde{M}}_{0}:=\|W_{0}\|_{B^{3/2}_{2,1}(\mathbb{R})\cap L^1(\mathbb{R})}$ is sufficiently small. Then it holds that
\begin{eqnarray}
\|\Lambda^{\ell}W(t,\cdot)\|_{L^2(\mathbb{R})}\lesssim \mathcal{\widetilde{M}}_{0}(1+t)^{-\frac{1}{4}-\frac{\ell}{2}} \label{R-E80}
\end{eqnarray}
for $0\leq\ell\leq 1/2$.
\end{cor}

\begin{rem}
The harmonic analysis allow to reduce significantly the regularity requirements on the initial data in
comparison with the previous works \cite{KRH,RW}, where $s\geq3$.
\end{rem}

\subsection{Timoshenko systems with equal speeds}
Consider the dissipative Timoshenko system, which is a set of two coupled wave equations of the form
\begin{equation}
\left\{
\begin{array}{l}
       \varphi_{tt}-(\varphi_x-\psi)_x=0,\\
       \psi_{tt}-\sigma(\psi_{x})_{x}-(\varphi_x-\psi)+\gamma \psi_t =0,
       \end{array}\right. \label{R-E81}
\end{equation}
which describes the transverse vibrations of a beam. Here $t\geq 0$ is the time variable, $x\in \mathbb{R}$ is the spacial
variable which denotes the point on the center line of the beam,
$\varphi(t,x)$ is the transversal displacement of the beam from an equilibrium state, and $\psi$ is the rotation
angle of the filament of the beam. The smooth function $\sigma(\eta)$ satisfies $\sigma'(\eta)>0$ for any $\eta\in\mathbb{R}$, and $\gamma$
is a positive constant, which is assumed to be one for simplicity.

System (\ref{R-E81}) is supplemented
with the initial data
\begin{equation}
(\varphi, \varphi_{t}, \psi, \psi_{t})(x,0)
=(\varphi_{0}, \varphi_{1}, \psi_{0}, \psi_{1})(x).\label{R-E82}
\end{equation}
The linearized system  of (\ref{R-E81}) reads correspondingly as
\begin{equation}
\left\{\begin{array}{l}
       \varphi_{tt}-(\varphi_x-\psi)_x=0,\\[2mm]
       \psi_{tt}-a^2\psi_{xx}-(\varphi_x-\psi)+\psi_t =0,
       \end{array}\right. \label{R-E83}
\end{equation}
with $a>0$ is the sound speed, which is defined by $a^2=\sigma'(0)$. The case $a=1$ corresponds to the Timoshenko system with equal wave speeds.
In this paper, we focus on the case.
The second author et al. \cite{IHK} introduced
the following quantities
\begin{equation}\label{R-E84}
v=\varphi_x-\psi, \quad
u=\varphi_t, \quad
z=\psi_x, \quad
y=\psi_t,
\end{equation}
so that  (\ref{R-E83}) can be rewritten as
\begin{equation}\label{R-E85}
\left\{\begin{array}{l}
        v_t-u_x+y=0,\\[2mm]
        u_t-v_x=0,\\[2mm]
        z_t-y_x=0,\\[2mm]
         y_t-z_x-v+y=0.
\end{array}\right.
\end{equation}
The initial data are given by
\begin{equation}\label{R-E86}
(v, u, z, y)(x, 0)
=(v_{0}, u_{0}, z_{0}, y_{0})(x),
\end{equation}
where $v_0=\varphi_{0,x}-\psi_{0}$, $y_0=\psi_1$,
$u_0=\varphi_1$ and $z_0=\psi_{0,x}$.
Furthermore, it was shown by \cite{IHK} that the dissipative structure of (\ref{R-E85})
is characterized by
\begin{equation}
{\rm Re}\,\lambda(i\xi)\leq -c\eta_1(\xi)
   \qquad {\rm for} \quad a=1,
\label{R-E87}
\end{equation}
for $c>0$, where $\lambda(i\xi)$ denotes the eigenvalues of the system
(\ref{R-E85}) in Fourier spaces and $\eta_1(\xi)=\frac{\xi^2}{1+\xi^2}$.

To state main results, it is convenient to rewrite (\ref{R-E81})-(\ref{R-E82}) as a Cauchy problem for
the hyperbolic form of first order
\begin{equation}\label{R-E88}
\left\{\begin{array}{l}
        V_t + A_{1}V_x + LV =p_{x},\\[2mm]
        V(x, 0) = V_0(x),
\end{array}\right.
\end{equation}
where
\begin{eqnarray*}
A_{1}=-\left(
\begin{array}{cccc}
0 & 1 & 0 & 0 \\
1 & 0 & 0 & 0 \\
0 & 0 & 0 & 1 \\
0 & 0 & 1 & 0
\end{array}
\right),\ \ \
L=\left(
\begin{array}{cccc}
0 & 0 & 0 & 1 \\
0 & 0 & 0 & 0 \\
0 & 0 & 0 & 0 \\
-1 & 0 & 0 & 1
\end{array}
\right).
\end{eqnarray*}
\begin{eqnarray*}
p=(0,0,0,p_{4})^{\top}
\end{eqnarray*}
with $p_{4}=\sigma(z)-\sigma(0)-z=O(z^2)$.

In \cite{IHK}, the second author et al. gave the decay property for linearized form (\ref{R-E85}) by Fourier energy methods, however, the energy functionals
used are not optimal, which leads to higher regularity needed to establish the global-in-time existence of smooth solutions to (\ref{R-E81})-(\ref{R-E82}), see \cite{IK}.
Recently, the optimal energy method was derived by a careful analysis for asymptotic expansions of the eigenvalues, see \cite{MK}. Racke and Said-Houari \cite{RS} strengthened decay properties in \cite{IHK} such that linearized solutions decay faster with a rate of $t^{-\gamma/2}$. Additionally,
the interested reader is referred to Timoshenko systems of other dissipative forms, for example, see \cite{ABMR} for memory-type dissipation case,
and \cite{FR} for thermal dissipation case.

Notice that $A_{1}$ is a real symmetric matrix, and the matrix $L$ is nonnegative definite but not symmetric,
so the Timoshenko system (\ref{R-E10}) is an example of hyperbolic systems
with non-symmetric dissipation, which lays outside the range of
the general study (see \cite{XK1}) for hyperbolic systems with the symmetric dissipation.
Recently, Mori, the first and second authors investigate the system in the framework of spatially Besov spaces and the result can be
drawn briefly as follows:
\begin{equation}
\left\{
\begin{array}{l} \mbox{The initial data} \
V_{0}\in B^{3/2}_{2,1}(\mathbb{R}),\\
\exists \delta_{0}>0\quad  \mbox{such that}\ \ I_{0}:=\|V_{0}\|_{B^{3/2}_{2,1}(\mathbb{R})}\leq \delta_{0},\\
 \end{array} \right.\label{R-E89}
\end{equation}
$\Rightarrow$
\begin{equation}
\left\{
\begin{array}{l}
V \in\widetilde{\mathcal{C}}(B^{3/2}_{2,1}(\mathbb{R}))\cap\widetilde{\mathcal{C}}^{1}(B^{1/2}_{2,1}(\mathbb{R})),\\
E(t)+D(t)\leq CI_{0},
 \end{array} \right.\label{R-E90}
\end{equation}
where
$$E(t):=\|V\|_{\widetilde{L}^\infty(B^{3/2}_{2,1}(\mathbb{R}))},$$
$$D(t):=\|y\|_{\widetilde{L}^2(B^{3/2}_{2,1})}+\|(v,z_{x})\|_{\widetilde{L}^2(B^{1/2}_{2,1})}
+\|u_{x}\|_{\widetilde{L}^2(B^{-1/2}_{2,1})}.$$
Clearly, we see that there is 1-regularity-loss phenomenon for the dissipation rates. Furthermore, under additional regularity assumption
$U_{0}\in \dot{B}^{-1/2}_{2,\infty}(\mathbb{R})$, we deduce the optimal decay rates of solution and its derivative of fraction order,
the interested author is referred to \cite{MXK}. Here, based on the new observations on the frequency-localization Duhamel principle in Section \ref{sec:3} as well as interpolation techniques, we can obtain desired  decay estimates if $U_{0}\in B^{3/2}_{2,1}(\mathbb{R})\cap\dot{B}^{-s}_{2,\infty}(\mathbb{R})$ with $s\in (1/4,1/2]$, which generalizes those results in \cite{MXK}.
\begin{thm}\label{thm5.4}
Let $V=(v, u, z, y)^{\top}$ be the global classical solution constructed by \cite{MXK}. Suppose that $V_{0}\in B^{3/2}_{2,1}(\mathbb{R})\cap \dot{B}^{-s}_{2,\infty}(\mathbb{R}) \ (1/4<s\leq1/2)$ and the norm
$\mathcal{M}_{0}:=\|V_{0}\|_{B^{3/2}_{2,1}(\mathbb{R})\cap \dot{B}^{-s}_{2,\infty}(\mathbb{R})}$ is sufficiently small. Then it holds that
\begin{eqnarray}
\|\Lambda^{\ell}V(t)\|_{X_{2}}\lesssim \mathcal{M}_{0}(1+t)^{-\frac{s+\ell}{2}} \label{R-E91}
\end{eqnarray}
for $0\leq\ell\leq1/2$, where $X_{2}:=B^{1/2-\ell}_{2,1}(\mathbb{R})$ if $0\leq\ell<1/2$ and $X_{2}:=\dot{B}^{0}_{2,1}(\mathbb{R})$ if $\ell=1/2$.
In particular, one has
\begin{eqnarray}
\|\Lambda^{\ell}V(t)\|_{L^2(\mathbb{R})}\lesssim \mathcal{M}_{0}(1+t)^{-\frac{s+\ell}{2}} \label{R-E92}
\end{eqnarray}
for $0\leq\ell\leq1/2$.
\end{thm}
\begin{proof}
The proof exactly follows from that of Thoerem \ref{thm5.2}, we feel free to skip it for brevity.
\end{proof}

Consequently, we have the analogue of Corollary \ref{cor5.2}.
\begin{cor}
Let $V=(v, u, z, y)^{\top}$ be the global classical solution constructed by \cite{MXK}. Suppose that $V_{0}\in B^{3/2}_{2,1}(\mathbb{R})\cap L^1(\mathbb{R})$ and the norm
$\mathcal{\widetilde{M}}_{0}:=\|V_{0}\|_{B^{3/2}_{2,1}(\mathbb{R})\cap L^1(\mathbb{R})}$ is sufficiently small. Then it holds that
\begin{eqnarray}
\|\Lambda^{\ell}V(t,\cdot)\|_{L^2(\mathbb{R})}\lesssim \mathcal{\widetilde{M}}_{0}(1+t)^{-\frac{1}{4}-\frac{\ell}{2}} \label{R-E922}
\end{eqnarray}
for $0\leq\ell\leq 1/2$.
\end{cor}

\section{Appendix}\setcounter{equation}{0}\label{sec:6}
For the convenience of reader, we would like to collect interpolation inequalities in $\mathbb{R}^{n}(n\geq1)$ related to Besov spaces, which
parallel the work \cite{SS}. However, we make some simplicity for use,
since their inequalities are related to the mixed spaces containing the microscopic velocity.

\begin{lem}\label{lem6.1}
Suppose $k\geq0$ and $m,\varrho>0$. Then the following inequality holds
\begin{eqnarray}
\|f\|_{\dot{B}^{k}_{2,1}}\lesssim \|f\|^{\theta}_{\dot{B}^{k+m}_{2,\infty}}\|f\|^{1-\theta}_{\dot{B}^{-\varrho}_{2,\infty}} \ \  \mbox{with}\ \ \ \theta=\frac{\varrho+k}{\varrho+k+m}. \label{R-E93}
\end{eqnarray}
\end{lem}

\begin{lem}\label{lem6.2}
Suppose $k\geq0$ and $m,\varrho>0$.  Then the following inequality holds
\begin{eqnarray}
\|\Lambda^{k}f\|_{L^2}\lesssim \|\Lambda^{k+m}f\|^{\theta}_{L^2}\|f\|^{1-\theta}_{\dot{B}^{-\varrho}_{2,\infty}} \ \  \mbox{with}\ \ \ \theta=\frac{\varrho+k}{\varrho+k+m}, \label{R-E94}
\end{eqnarray}
where (\ref{R-E94}) is also true for $\partial^{\alpha}$ with $|\alpha|=k\ (k$ nonnegative integer).
\end{lem}

\begin{lem}\label{lem6.3}
Suppose that $m\neq\varrho$. Then the following inequality holds
\begin{eqnarray}
\|f\|_{\dot{B}^{k}_{p,1}}\lesssim \|f\|^{1-\theta}_{\dot{B}^{m}_{r,\infty}}\|f\|^{\theta}_{\dot{B}^{\varrho}_{r,\infty}}, \label{R-E95}
\end{eqnarray}
where $0<\theta<1$,  $1\leq r\leq p\leq\infty$ and $$k+n\Big(\frac{1}{r}-\frac{1}{p}\Big)=m(1-\theta)+\varrho\theta.$$
\end{lem}

\begin{lem}\label{lem6.4}
Suppose that $m\neq\varrho$. One has the interpolation inequality of Gagliardo-Nirenberg-Sobolev type
\begin{eqnarray}
\|\Lambda^{k}f\|_{L^{q}}\lesssim \|\Lambda^{m}f\|^{1-\theta}_{L^{r}}\|\Lambda^{\varrho}f\|^{\theta}_{L^{r}}, \label{R-E96}
\end{eqnarray}
where $0\leq\theta\leq1$,  $1\leq r\leq q\leq\infty$ and $$k+n\Big(\frac{1}{r}-\frac{1}{q}\Big)=m(1-\theta)+\varrho\theta.$$
\end{lem}

\begin{lem}\label{lem6.5}
Suppose that $\varrho>0$ and $1\leq p<2$. It holds that
\begin{eqnarray}
\|f\|_{\dot{B}^{-\varrho}_{r,\infty}}\lesssim \|f\|_{L^{p}} \label{R-E97}
\end{eqnarray}
with $1/p-1/r=\varrho/n$. In particular, this holds with $\varrho=n/2, r=2$ and $p=1$.
\end{lem}

\section*{Acknowledgments}
J. Xu is partially supported by the National
Natural Science Foundation of China (11471158), the Program for New Century Excellent
Talents in University (NCET-13-0857) and the NUAA Fundamental
Research Funds (NS2013076).  The work is also partially supported by
Grant-in-Aid for Scientific Researches (S) 25220702 and (A) 22244009.


\begin{thebibliography}{99}
\bibitem {ABMR}
F. Ammar Khodja, A. Benabdallah, J. E. Mu\~{n}oz Rivera and R. Racke, Energy
decay for Timoshenko systems of memory type, \textit{J. Differential Equations} {\bf{194}} (2003)
82--115.

\bibitem {BCD}
H.~Bahouri, J.~Y.~Chemin and R.~Danchin, \textit{Fourier Analysis
and Nonlinear Partial Differential Equations}, Springer-Verlag:
Berlin/Heidelberg, 2011.

\bibitem {BHN}
S.~Bianchini, B.~Hanouzet and R.~Natalini, Asymptotic
behavior of smooth solutions for partially dissipative hyperoblic
systems with a convex entropy, \textit{Comm. Pure Appl. Math.}
{\bf{60}} (2007) 1559--1622.

\bibitem {CL} J.-Y. Chemin and N. Lerner, Flot de champs de vecteurs non lipschitziens et \'{e}quations de Navier-Stokes,
\textit{J. Differential Equations} {\bf{121}} (1995) 314--328.

\bibitem {CG} J.-F. Coulombel and T. Goudon, The strong relaxation
limit of the multidimensional isothermal Euler equations, \textit{Trans.
Amer. Math. Soc.} {\bf{359}} (2007) 637--648.

\bibitem{CLL}
G.-Q.~Chen, C. D. Levermore and  T.-P. Liu, Hyperbolic
conservation laws with stiff relaxation terms and entropy.
\textit{Comm. Pure Appl. Math.} {\bf{47}} (1994) 787--830.

\bibitem{DH}
C.~M. Dafermos and L. Hsiao, Development of singularities in solutions of the equations of nonlinear thermoelasticity, \textit{Quart. Appl. Math.},
{\bf{44}} (1986) 463--474.

\bibitem {FL}
K. O.Friedrichs and P. D. Lax, Systems of conservation equations
with a convex exttension. \textit{Proc. Nat. Acad. Sci. USA}
{\bf{68}} (1971) 1686--1688.

\bibitem {FR}
H. D. Fern\'{a}ndez Sare and R. Racke, On the stability of damped Timoshenko systems:
Cattaneo vs. Fourier law, \textit{Arch. Rational Mech. Anal.} {\bf{194}} (2009) 221--251.

\bibitem {G}
S. K.~Godunov, An interesting class of quasilinear systems.
\textit{Dokl. Akad. Nauk SSSR} {\bf{139}} (1961) 521--523.

\bibitem {HL}
L. Hsiao and T.-P. Liu. Convergence to nonlinear diffusion waves for
solutions of a system of hyperbolic conservation laws with damping,
\textit{Comm. Math. Phys.} {\bf{143}} (1992) 599--605.

\bibitem {HP2} F. Huang and R. Pan, Convergence rate for compressible Euler equations with damping
and vacuum, \textit{Arch. Rational Mech. Anal.} {\bf{166}} (2003) 359--376.

\bibitem {HMP} F. Huang, P. Marcati and R. Pan,
Convergence to Barenblatt solution for the compressible Euler
equations with damping and vacuum, \textit{Arch. Rational Mech. Anal.}  {\bf{176}}
(2005) 1--24.

\bibitem{HN}
B.~Hanouzet and  R.~Natalini, Global existence of smooth
solutions for partially disipative hyperbolic systems with a convex
entropy. \textit{Arch. Rational Mech. Anal.} {\bf{169}}(2003) 89--117.

\bibitem{HZ}
D.~Hoff and K. Zumbrun, Multi-dimensional diffusion waves for the Navier-Stokes equations
of compressible flow, \textit{Indiana Univ. Math. J.} {\bf{44}} (1995) 603--676.

\bibitem{HM}
W.~J.~Hrusa and S.~A.~Messaoudi, On formation of singularities on one-dimensional nonlinear thermoelasticity,
\textit{Arch. Rational Mech. Anal.} {\bf{111}}(1990) 135--151.

\bibitem{HR}
Y. Hu and R. Racke, Formation of singularities in one-dimensional thermoelasticity with second sound, \textit{Quart. Appl. Math.} {\bf{72}} (2014) 311--321.

\bibitem{IHK}
K. Ide, K. Haramoto and S. Kawashima, Decay property of regularity-loss type for
dissipative Timoshenko system, \textit{Math. Models Meth. Appl. Sci.} {\bf{18}} (2008) 647--667.

\bibitem{IK}
K. Ide and S. Kawashima, Decay property of regularity-loss type and nonlinear effects for dissipative
Timoshenko system, \textit{Math. Models Meth. Appl. Sci.}  {\bf{18}} (2008) 1001--1025.

\bibitem {JR} S. Junca and M. Rascle, Strong relaxation of the
isothermal Euler system to the heat equation, \textit{Z. Angew. Math. Phys.}
 {\bf{53}} (2002) 239--264.

\bibitem{K}
T.~Kato, The Cauchy problem for quasi-linear
symmetric hyperbolic systems, \textit{Arch. Rational Mech. Anal.}
{\bf{58}} (1975) 181--205.

\bibitem{Ka}
S.~Kawashima, \textit{Systems of a hyperbolic-parabolic composite type, with applications to the equations of magnetohydrodynamics},
Doctoral Thesis, Kyoto University, 1984. http://repository.kulib.kyoto-u.ac.jp/dspace/handle/2433/97887

\bibitem{KRH}
A. Kasimov, R. Racke and B.S.Houari,
Global exsitence and decay of solutions of the Cauchy problem in thermoelasticity with second sound, \textit{Appl. Anal.},
{\bf{93}} (2014) 911--935.

\bibitem{KY} S.~Kawashima and W.-A.~Yong, Dissipative structure and
entropy for hyperbolic systems of balance laws, \textit{Arch.
Rational Mech. Anal.} {\bf{174}} (2004) 345--364.

\bibitem{KY2} S.~Kawashima and W.-A.~Yong, Decay estimates for
hyperbolic balance laws, \textit{J. Anal. Appl.}  {\bf{28}} (2009) 1--33.

\bibitem {LC}
C. J. Lin and J.-F. Coulombel, The strong relaxation limit of the
multidimensional compressible Euler equations, \textit{Nonlinear Differ.
Equs. Appl.}  {\bf{20}} (2013) 447--461.

\bibitem{M} A.~Majda, \textit{Compressible Fluid Flow and
Conservation laws in Several Space Variables}, Springer-Verlag:
Berlin/New York, 1984

\bibitem {MM} P. Marcati and A. Milani, The one-dimensional Darcy's
law as the limit of a compressible Euler flow, \textit{J. Differential Equations} {\bf{84}}
(1990) 129--147.

\bibitem {MR} P. Marcati and B. Rubino, Hyperbolic to parabolic
relaxation theory for quasilinear first order systems, \textit{J. Differential Equations} {\bf{162}} (2000) 359--399.

\bibitem{Ma}
A.~Matsumura,\ \textit{An energy method for the equations of motion of
compressible viscous and heat-conductive fluids}, MRC Technical
Summary Report, Univ. of Wisconsin-Masison, $\sharp$ \textbf{2194}
(1981)

\bibitem{MK}
N. Mori and S. Kawashima, Decay property for the Timoshenko system with Fourier's type heat conduction,
\textit{J. Hyper. Differ. Equs.} {\bf{11}} (2014) 135--157.

\bibitem{MXK}
N. Mori, J. Xu and S. Kawashima, Global existence and optimal decay rates for the Timoshenko system: the case of equal wave speeds, \textit{J. Differential Equations}, to appear, 2014.

\bibitem {N2} T. Nishida, Nonlinear hyperbolic equations and
relates topics in fluid dynamics, \textit{Publ. Math. D'Orsay} (1978) 46--53.

\bibitem {N}
K. J. Nishihara, Convergence rates to nonlinear difusion eaves for solutions of system of hyperbolic conservation laws with damping,
\textit{J. Differential Equations} {\bf{131}} (1996) 171--188.

\bibitem {NWY}
K. J. Nishihara, W.K.Wang and T. Yang, $L^p-$convergence rate to nonlinear diffusion waves for $p-$system with damping,
\textit{J. Differential Equations} {\bf{161}} (2000) 191--218.

\bibitem {NY}
K. J. Nishihara and T. Yang, Boundary effect on asymptotic behavior
of solutions to the p-system with linear damping, \textit{J. Differential Equations} {\bf{156}}
(1999) 439--458.

\bibitem {RS1}
R. Racke and B. Said-Houari, Decay rates and global existence for semilinear dissipatie Timoshenko systems,
\textit{Quart. Appl. Math.} {\bf{71}} (2013) 229--266.

\bibitem{RS} T. Ruggeri and D.Serre, Stability of constant equilibrium state for dissipative balance laws
system with a convex entropy. \textit{Quart. Appl. Math.} {\bf{62}} (2004) 163--179.

\bibitem{RW}
R. Rackes and Y.G.Wang, Nonlinear well-posedness and rates of decay in thermoelasticity with second sound, \textit{J. Hyper. Diff. Equs.}
 {\bf{5}} (2008) 25--43.

\bibitem{SK}
Y.~Shizuta and S.~Kawashima, Systems of equations of
hyperbolic-parabolic type with applications to the discrete
Boltzmann equation, \textit{Hokkaido Math. J.}  {\bf{14}} (1985) 249--275.

\bibitem{SS}
V.~Sohinger and R. M.~Strain, The Boltzmann equation, Besov spaces, and optimal time decay rates in $\mathbb{R}^{n}_{x}$, \textit{Adv. Math.}
{\bf{261}} (2014) 274--332.

\bibitem {STW}
T.~Sideris, B.~Thomases and D. H.Wang, Long time behavior of
solutions to the 3D compressible Euler with damping, \textit{Commun. Part. Diff. Equs.} {\bf{28}} (2003) 953--978.

\bibitem {T}
M.~A. Tarabek, On the existence of smooth solutions in one-dimensional nonlinear thermoelasticity with second sound, \textit{Quart. Appl. Math.},
{\bf{50}} (1992) 727--742.

\bibitem {TW1}
Z.~Tan and Y.~Wang, Global solution and large-time behavior of the 3D compressible Euler equations with damping,
\textit{J. Differential Equations} {\bf{254}} (2013) 1686--1704.


\bibitem {TW2}
Z.~Tan and G.~C.~Wu, Large time behavior of solutions for compressible Euler equations with damping in $\mathbb{R}^{3}$, \textit{J. Differential Equations}, {\bf{252}} (2012) 1546--1561.

\bibitem{UKS}
T.~Umeda,  S.~Kawashima and Y.~Shizuta, On the decay of solutions to the linearized equations of electro-magneto-fluid dynamics, \textit{Japan J. Appl. Math.} {\bf{1}} (1984) 435-457.

\bibitem{W}
J.~H.~Wu, Lower bounds for an integral involving fractional
Laplacians and the generalized Navier-Stokes equations
in Besov spaces, \textit{Commun. Math. Phys.} {\bf{263}} (2005) 803--831.

\bibitem{WY}
W. Wang and T. Yang,  The pointwise estimates of solutions for Euler
equations with damping in multi-dimensions. \textit{J. Differential Equations}
{\bf{173}} (2001) 410--450.

\bibitem{X}
J.Xu, Strong relaxation limit of multi-dimensional isentropic Euler
equations, \textit{Z. Angew. Math. Phys.}, {\bf{61}} (2010) 389--400.

\bibitem{XK1}
J. Xu and S. Kawashima, Global classical solutions for partially
dissipative hyperbolic system of balance laws, \textit{Arch. Rational Mech. Anal.} {\bf{211}} (2014) 513--553.

\bibitem{XK2}
J. Xu and S. Kawashima, The optimal decay estimates on the framework of
Besov spaces for generally dissipative systems, preprint (2014).

\bibitem{XK3}
J. Xu and S. Kawashima, Diffusive relaxation limit of classical solutions to the damped compressible Euler equations,
\textit{J. Differential Equations} {\bf{256}}, 771--796 (2014)


\bibitem{XW}
J. Xu and Z. J. Wang, Relaxation limit in Besov spaces for
compressible Euler equations, \textit{J. Math. Pures Appl.}, {\bf{99}} (2013) 43--61.

\bibitem{YW}
L. Yang and Y.G.Wang, $L^p$-$L^q$ decay estimates for the Cauchy problem of linear thermoelastic systems with second sound in one space varibale,
\textit{Quart. Appl. Math.} {\bf{64}} (2006) 1--15

\bibitem{Y1}
W.-A. Yong, An interesting class of partial differential equations, J. Math. Phys. {\bf{49}} (2008), 033503.

\bibitem{Y} W.-A.Yong, Entropy and global existence for hyperbolic balance laws, \textit{Arch.
Rational Mech. Anal.} {\bf{172}} (2004) 247--266

\end{thebibliography}
\end{document}